\documentclass[12pt, reqno]{amsart}
\usepackage{amsmath, amsthm, amscd, amsfonts, amssymb, graphicx, xcolor, mathtools}
\usepackage{subcaption}

\usepackage[bookmarksnumbered, colorlinks, plainpages,
citecolor=teal,
linktoc=all]{hyperref}

\usepackage{tikz}
\usetikzlibrary{patterns}
\usepackage{comment} 

\usepackage{extarrows} 
\usepackage{array} 
\usepackage{enumitem}

\usepackage[table]{xcolor} 

\makeatletter
\def\l@section{\@tocline{1}{0pt}{0pc}{}{}}
\def\l@subsection{\@tocline{2}{0pt}{1.5pc}{}{}}
\makeatother

\newtheorem{theorem}{Theorem}[section]
\newtheorem{lemma}[theorem]{Lemma}
\newtheorem{proposition}[theorem]{Proposition}
\newtheorem{corollary}[theorem]{Corollary}
\newtheorem{definition}[theorem]{Definition}
\newtheorem{notation}[theorem]{Notation}
\newtheorem{example}[theorem]{Example}
\newtheorem{construction}[theorem]{Construction}

\theoremstyle{remark}
\newtheorem{remark}[theorem]{Remark}
\numberwithin{equation}{section}

\newcommand{\RR}{\mathbb{R}}
\newcommand{\Rpos}{\mathbb{R}_+}
\newcommand{\Rnz}{\mathbb{R}^*}

\newcommand{\Zpos}{\mathbb{Z}_+}

\newcommand{\defset}[2]{ \left\{\ #1 \ \left\lvert\ #2 \right.\ \right\} }

\newcommand{\absbig}[1]{ \left| #1 \right| }
\newcommand{\normbig}[1]{ \left \Vert #1 \right \Vert }
\newcommand{\ol}{\overline}

\newcommand{\ortprehomeo}{orientation-preserving homeomorphism}

\DeclareMathOperator{\wt}{W}
\DeclareMathOperator{\Tree}{Tree}
\DeclareMathOperator{\Fore}{Fore}
\DeclareMathOperator{\Lab}{\mathfrak{L}}
\DeclareMathOperator{\comb}{Comb}

\newcommand{\pp}{\Xi}
\newcommand{\ppf}{\Xi_F}
\newcommand{\Ptt}{\mathfrak{P}}
\newcommand{\ptt}[1]{\mathfrak{#1}}
\newcommand{\perm}{\underline{P}}
\newcommand{\ee}{\mathrm{e}}

\newcommand{\one}{\mathbf{1}}

\newcommand{\PP}{\mathbf{P}}
\newcommand{\TT}{\mathbf{T}}
\newcommand{\FF}{\mathbf{F}}

\newcommand{\PT}{\mathbf{P} T}
\newcommand{\FM}{\mathbf{F} M}

\newcommand{\TM}{\mathbf{T} M}
\newcommand{\TR}{\mathbf{T} R}

\tikzset{
    blk/.style = {draw, circle, fill=black, text=white},
    wht/.style = {draw, circle, fill=white, text=black},
    mid_auto/.style = {midway, auto}
}

\begin{document}

\title[Enumeration of weighted plane trees]{Enumeration of weighted plane trees\\ by a permutation model} 

\author[S.C. Lu, S. Yi]{Sicheng Lu and Yi Song}

\subjclass[2020]{Primary 05C30 05A19; Secondary 05C10 57M15.} 

\keywords{plane tree, enumeration, permutation, bijection}

\date{%
}

\begin{abstract} 
This manuscript addresses an enumeration problem on weighted bi-colored plane trees with prescribed vertex data, with all vertices labeled distinctly. 
We give a bijection proof of the enumeration formula originally due to Kochetkov \cite{YYK15}, hence affirmatively answer a question of Adrianov-Pakovich-Zvonkin in \cite{APZ2020}. The argument is purely combinatorial and totally constructive, remaining valid for real-valued edge weights. 
A central process is a geometric construction that directly encodes each tree as a permutation. 
We also exhibit algebraic relationships between the enumeration problem, the partial order on partitions of vertices and the Stirling numbers of the second kind. 
Some computation examples are presented as appendices.
\end{abstract} 

\maketitle

{\footnotesize
\tableofcontents
}

\section{Introduction}\label{sec:intro}

As combinatorial objects, embedded graphs on surfaces are closely related to a variety of topics, such as branched coverings between Riemann surfaces, representation theory of symmetry groups, moduli space of algebraic curves, matrix integrals and many others \cite{Petr20, KonZor03, Kont92}. We refer to \cite{LsZak04} for a detailed exposure, also as a comprehensive literature for basic concepts in this manuscript. 
Among these, plane trees are the simplest and have already generated a substantial body of work. 
This manuscript focuses on a class of weighted bi-colored plane trees and their enumeration.

\medskip
\subsection{Labeled weighted bi-colored plane tree}\label{ssec:LWBP_tree}
We begin by introducing the required definitions and reviewing previous results, then proceed to the enumeration formula. 

A weighted bi-colored plane tree (\textbf{WBP-tree} in short) consists of 
    \begin{itemize}
        \item an embedded tree $(V,E)$ in $\RR^2$, with vertex set $V$ and edge set $E$; 
        \item a partition of $V$ into \textbf{black vertices} $V^+$ and \textbf{white vertices} $V^-$, such that each edge connects a black vertex to a white vertex;  
        \item a \textbf{weight function} $\wt_E : E \to \Rpos$ on edges.
    \end{itemize}

The weight function $W_E$ naturally extends to a function on $V$ by defining
\[ \wt_V(v^\pm):=\pm \sum_{e \in E(v^\pm)} \wt_E(e),\quad \forall v^\pm \in V^\pm \ . \]

Here $E(v)$ is the set of edges adjacent to $v$, and the sign is used to distinguish black and white vertices. 
The distribution of weights over vertices is recorded as a data called passport. 
Here we are using a rather general definition, similar to the settings in \cite{GorTor1980}. 
\begin{definition}\label{def:passport}
    A \textbf{passport} $\pp = (S,\lambda,\wt)$ consists of 
    \begin{itemize}
        \item a nonempty finite set $S$ called the \textbf{index set}; 
        \item a \textbf{multiplicity function} $\lambda: S \to \Zpos$; and  
        \item a \textbf{weight function} $\wt : S\to \Rnz$ such that
        \begin{equation}\label{eq:passport_weight}
            \sum_{s\in S} \lambda(s) \cdot \wt(s) =0 \ .
        \end{equation}
    \end{itemize}
    Also define
    \begin{equation*}
        p:= \sum_{\wt(s)>0} \lambda(s), \quad q:= \sum_{\wt(s)<0} \lambda(s),\ \quad |\pp|:=p+q, \quad \Vert \pp \Vert:= \frac12 \sum_{s\in S} \lambda(s)\cdot\absbig{\wt(s)}
    \end{equation*}
    to be the number of black vertices, white vertices, total vertices and total weights respectively. 
\end{definition}

\begin{definition} \label{def:labeled_tree}
    A \textbf{labeled weighted bi-colored plane tree} (\textbf{LWBP-tree} in short) of passport $\pp = (S,\lambda,\wt)$ is a WBP-tree $(V,E,\wt_E)$ with a surjection $\Lab: V \to S$ called \textbf{labeling}, such that
    \begin{enumerate}
        \item as weight function on $V$, $\wt \circ \Lab = \wt_V$; 
        \item as multiplicity function on $S$, $|\Lab^{-1} (s)|= \lambda(s)$ for all $s\in S$. 
    \end{enumerate} 
    In this manuscript, we use $T = (V, E, \wt_E, \Lab)$ to denote an LWBP-tree, and all trees are LWBP-trees. 
\end{definition}
The labeling distinguishes vertices of the same weight. Such formalization is useful in the later enumeration problems. See Example \ref{eg:tree_label} for an illustration, after our notations explained. 

\medskip
As finite combinatorial objects, two related questions arise naturally: 
\begin{center}
    \emph{Is there a labeled weighted bi-colored plane tree of a prescribed passport? \\
    If so, can we enumerate all such trees?} 
\end{center}
Here trees are counted up to isomorphisms. Two LWBP-trees $T_1, T_2$ of passport $\pp$ are said to be \textbf{isomorphic} if there exists an \ortprehomeo\ $I:\RR^2 \to \RR^2$ such that $I(V^{\pm}_1) = V^{\pm}_2$ , $I(E_1) = E_2$ as elements of graphs, $\wt_{E_2} \circ\ I = \wt_{E_1}$ as weight functions, and $\Lab_2 \circ\ I = \Lab_1$ as labeling. 
Denote $\Tree(\pp)$ as the set of all non-isomorphic LWBP-trees of passport $\pp$. 

\medskip
Usually we consider passports with injective weight function $\wt$. Such passports are called \textbf{trivial}, and the multiplicity of an index $s\in S$ just counts the number of vertices of weight $\wt(s)\in \RR^*$. Then a trivial passport is simply two partitions of the total weight $\Vert \pp \Vert$. Boccara \cite{Boc82} and Zannier \cite{Zan95} independently indicate the necessary and sufficient condition for the existence, when all the weights are integers. 
Following notations in Definition \ref{def:passport}, let 
\[ \gcd(\wt) := \gcd\{ \absbig{\wt(s)}, s \in S \} \ . \]
Then there exists a tree of passport $\pp$ if and only if 
\begin{equation}\label{eq:exist_WBPTree}
    (p + q - 1) \cdot \gcd(\wt) \leq \Vert \pp \Vert \ .
\end{equation}
It is not difficult to see that the same criterion holds for all passports with integer weights. See Remark \ref{rmk:real_weighted} for real-valued cases. 

The enumeration problem for trivial passport has been considered in various forms and we list some of them. Tutte \cite{Tutte1970} studies a slightly different case, in which the weight on an edge is allowed to be zero, using generating functions. 
Zvonkin \cite{Zvo13} considers enumeration problem with a given number of black and white vertices, together with the total integer weight over all edges. 
Pakovich-Zvonkin \cite{PakZvo14} classifies all passports with integer weights such that $\absbig{\Tree(\pp)} = 1$.

By comparison, results for non-weighted bi-colored plane trees are more plentiful and specific. This is equivalent to require the weight function on edges to be constantly 1.
The enumeration formula according to passport is already given by Harary-Tutte \cite{HarTut64}. 
Goulden-Jackson \cite{GlJac92} relates enumeration of such trees and similar objects with appropriate characteristics to connection coefficients for the symmetric group. 
Goupil-Schaeffer \cite{GpSch98} and Poulalhon-Schaeffer \cite{PouSch02} generalize the above results to surfaces of arbitrary genus. 
Chapuy-F\'eray-Fusy \cite{CFF13} later provides a combinatorial proof of Goupil-Schaeffer's result.

\bigskip
\subsection{Kochetkov's enumeration formula} 
There is another choice of passport. We call a passport \textbf{full}, denoted by $\ppf$, if $\lambda(s) \equiv 1$ for all $s\in S$ (also called \emph{simple} in \cite{YYK15} and \emph{totally labelled} in \cite{APZ2020}). The multiplicity function of a full passport will be denoted as $\one$. Hence for a LWBP-tree of a full passport, its labeling $\Lab:V\to S$ is a bijection that makes all vertices distinct. 

Kochetkov \cite{YYK15} presents an enumeration formula for trees of full passport, which will be stated below. This manuscript aims to present a bijection proof of Kochetkov's formula. Thus, our approach affirmatively answers the last part of problem ``\textsc{Enumeration according to a passport}'' in \cite[Chapter 12]{APZ2020} by Adrianov-Pakovich-Zvonkin. 
Besides, it is combinatorial and constructive, naturally feasible for real-valued weights, hence providing a theoretical algorithm for finding all weighted bi-colored plane trees of a prescribed full passport. 

More notations are needed for the statement of formula.

\begin{definition}\label{def:partition}
Let $\ppf = (S, \one, \wt)$ be a full passport.
\begin{enumerate}[label=\arabic*.]
    \item 
     For a non-empty subset $S' \subset S$, denote $\pp(S'):=(S', \one, \wt')$ as the full passport with $\wt'$ being the restriction of $\wt$ on $S'$. $\pp(S')$ is also called a \textbf{subpassport} of $\ppf$.  

    \item 
    An \textbf{n-partition} of $\ppf$ is a set of subpassports $\ptt{p} := \{ \pp(S_i) \}_{i=1}^{n}$ such that $\bigsqcup_{i=1}^{n} {S_i} = S$ and each $S_i$ is non-empty.
    \begin{enumerate}[label=\roman*.]
        \item $|\ptt{p}|:=n$ is called the \textbf{length} of the partition. 
        \item Define the quantity
        \[ X(\ptt{p}) := \prod_{i=1}^{n} \left( \absbig{ S_i } - 1 \right)! \ . \]
        \item $\Ptt(\ppf)$ and $\Ptt_n(\ppf)$ are the set of all partitions and $n$-partitions of $\ppf$. 
        \item The trivial 1-partition of $\ppf$ is denoted by $\ptt{e} := \{\ppf=\pp(S)\}$.
        \item The \textbf{maximum partition length} of $\ppf$ is denoted by \[ m(\ppf) := \max\defset{\absbig{\ptt{p}}}{\ptt{p} \in \Ptt(\ppf)} \ . \]
    \end{enumerate}

    \item 
    $\ppf$ is \textbf{non-decomposable} if $\Ptt(\ppf) = \{ \ptt{e} \}$, or equivalently $m(\ppf)=1$ (also called \emph{non-separable} in \cite{APZ2020}). Otherwise, it is \textbf{decomposable}.
\end{enumerate}
\end{definition}

Here is Kochetkov's formula in \cite{YYK15}, which is the main topic of this manuscript. \begin{theorem}\label{thm:passport_simple} 
    Let $\ppf$ be a full passport, then 
    \begin{equation}\label{eq:enumeration_passport_decomposable}
        \absbig{\Tree(\ppf)} = \sum_{\ptt{p} \in \Ptt(\ppf)} (-1)^{|\ptt{p}|-1} (\absbig{\ppf}-1)^{|\ptt{p}|-2} X(\ptt{p}) \ .
    \end{equation}
    In particular, when $\ppf$ is non-decomposable, 
    \begin{equation} \label{eq:enumeration_passport_nondecomposable}
        \absbig{\Tree(\ppf)} = (\absbig{\ppf} - 2)! \ .
    \end{equation}   
\end{theorem} 

\bigskip
The original proof of Kochetkov in \cite{YYK15}, also rearranged in \cite[\S 11.4]{APZ2020}, transforms each LWBP-tree of non-decomposable full passport into a meromorphic function on the complex plane with prescribed information of the zeros and poles. Then the enumeration problem reduces to counting all possible coefficients of a rational function, which are constrained by a system of algebraic equations. Since most equations are non-linear, some algebraic geometry style arguments are used to enumerate solutions to these equations. As a consequence, such proof is highly non-constructive. 
In \cite[Chapter 12]{APZ2020}, Adrianov-Pakovich-Zvonkin ask for ``a bijective proof of the formula'' for non-decomposable case (see the problem ``\textsc{Enumeration according to a passport}''). Our approach gives an affirmative answer. Besides, it is combinatorial and constructive, naturally feasible for real-valued weights, hence providing a theoretical algorithm for finding all weighted bi-colored plane trees of a prescribed full passport. 

For the decomposable case, both proofs in \cite{YYK15} and here use some combinatorial tricks to relate it to non-decomposable case. 
Adrianov-Pakovich-Zvonkin also ask, in the same question of \cite{APZ2020}, for an explicit formula for decomposable case ``without resorting to inclusion-exclusion''. We give some response in Remark \ref{rmk:proof_without_in-ex}.

\bigskip
\subsection{Outline of the proof} 
Here is a summary of our approach. From now on, \emph{all passports are full by default.} 

As intermediate steps, we need the concept of \textbf{forest}, which is a disjoint union of trees, possibly disconnected. Then we can talk about LWBP-forests and the set $\Fore(\ppf)$, by analogy with Definition \ref{def:labeled_tree}. 
Since each tree is contractable as a subset on plane, for any LWBP-forest, the relative position of different connected components does not affect its isomorphic class. 

\medskip
Let $\ppf=(S,\one,\wt)$ be a full passport with $N :=\absbig{S} = \absbig{\ppf} \geq 2$. Since the labeling of a LWPB-tree of full passport is a bijection from its vertex set to $S$, for convenience, we shall identify both the index set $S$ and the vertex set $V$ of any tree $T \in \Tree(\ppf)$ with  
\[ [N]:=\{1,2,\cdots,N\} \ . \]

\begin{definition}\label{def:twice_marked_tree}
    Consider the following sets, with $a,b\in[N]$ regarded as vertices: 
\begin{equation}\begin{split}
    &\FM(\ppf) := \defset{\left( G; a, b \right)}{G \in \Fore(\ppf),\ a \neq b \in [N]} \ , \\
    &\TM(\ppf) := \defset{\left( T; a, b \right)}{T \in \Tree(\ppf),\ a \neq b \in [N]} \ , \\
    &\TR(\ppf) := \defset{\left( T;a,b \right)}{
    \begin{array}{c}
          a,b \textrm{ are the black and white} \\
          \textrm{vertices of an edge in } T
    \end{array} 
    } \subset \TM(\ppf) \ . 
\end{split}\end{equation}
$\FM(\ppf)$ and $\TM(\ppf)$ are called the set of \textbf{twice-marked} forests and trees of passport $\ppf$. $\TR(\ppf)$ is called the set of \textbf{rooted trees} of passport $\ppf$. 
\end{definition}

\medskip
\begin{definition}
    Let $\PP(\ppf)$ be the set of all permutations of elements in $S=[N]$. A permutation $\perm\in\PP(\ppf)$ is denoted as a series $\perm=(s_1, s_2, \cdots, s_N)$. 
\end{definition}

The core of our argument is a geometric process called \emph{combing map} (Definition \ref{def:comb}) that turns any permutation into a twice-marked forest (see Figure \ref{fig:combing} for a preview).  
This map restricts to a bijection $$ \TR(\ppf) \ \cong \ \PT_+(\ppf) $$ between the set of rooted trees and an explicitly characterized subset $\PT_+(\ppf)$ of $\PP(\ppf)$ (Definition \ref{def:pos_perm}, Proposition \ref{prop:pos_and_tree}). 
The enumeration is then settled by an explicit formula for $|\PT_+(\ppf)|$ (Theorem \ref{thm:number_PosT}). 
\begin{proof}[Proof of Theorem \ref{thm:passport_simple}]
\[
    \absbig{\Tree(\ppf)} = \frac{\absbig{\TR(\ppf)}}{\absbig{\ppf}-1} \xlongequal[\ref{prop:pos_and_tree}]{\textrm{Prop.}} \frac{|\PT_+(\ppf)|}{|\ppf|-1} \xlongequal[\ref{thm:number_PosT}]{\textrm{Thm.}} \sum_{\ptt{p}\in\Ptt(\ppf)}  \frac{(-\absbig{\ppf} + 1)^{\absbig{\ptt{p}} - 1}}{|\ppf|-1} X(\ptt{p}) \ .
\]
\end{proof}

The formula for $|\PT_+(\ppf)|$ is obtained by strong induction on the maximal partition length $m(\ppf)$, hence the non-decomposable case initiates the induction. 
The combing map actually yields a larger bijection
$$ \TM(\ppf) \ \cong \ \PT(\ppf) $$
between the set of twice-marked trees and a set $\PT(\ppf)$ containing $\PT_+(\ppf)$ (Definition \ref{def:PT}, Theorem \ref{thm:comb_biject}). 
When $\ppf$ is non-decomposable, $\PT(\ppf)$ coincides with $\PP(\ppf)$, so \eqref{eq:enumeration_passport_nondecomposable} is immediate, and formula for $|\PT_+(\ppf)|$ follows from the smaller bijection. 

When $\ppf$ is decomposable, this coincidence no longer occurs, and we enumerate $\PT_+(\ppf)$ recursively. 
By studying the partial order on the set of partitions, utilizing the inversion formula for Stirling number of the second kind, together with many other ingredients, we derive the desired formula.

\medskip

\begin{remark}\label{rmk:real_weighted}
    Bi-colored plane trees with real weights are equivalent to a special class of meromorphic differentials on Riemann sphere $\overline{\mathbb{C}}$. 
    Following notations in Definition \ref{def:passport}, there exists an LWBP-tree of trivial passport $\pp$, if and only if there exists a meromorphic 1-form on $\overline{\mathbb{C}}$ satisfying that
    \begin{itemize}
        \item it has a unique zero of order $p+q-2$ and $(p+q)$ simple poles;
        \item the integral along any closed path avoiding the zero and poles is real;
        \item there are exactly $\lambda(s)$ simple poles of residue $\wt(s)$ for all $s\in S$. 
    \end{itemize}
    Also see \cite{SXY22axv} for a similar statement. A branched cover between two Riemann spheres with three branched points in \cite{SXY22axv} will correspond to a bi-colored plane tree with integer weights. 
    
    The existence of meromorphic differentials on Riemann surfaces is completely known (for example, see \cite{GenTah21}). With these results, \eqref{eq:exist_WBPTree} can be generalized as follows. 
    If there exists $L>0$ such that $L \absbig{\wt(s)} \in \Zpos$ for all $s\in S$, then 
    \[ \gcd(\wt) := \frac{1}{L} \cdot \gcd\{ L \absbig{\wt(s)}, s \in S \} \]
    is well-defined; otherwise, $\gcd(\wt):=0$. 
    Then there exists a tree with real-valued weights of trivial passport $\pp$ if and only if 
    \begin{equation}\label{eq:exist_WBPTree_realW}
        (p + q - 1) \cdot \gcd(\wt) \leq \Vert \pp \Vert \ .
    \end{equation}
    Meromorphic differential is a whole other story, so we will not expand on this topic here. \hfill $\square$
\end{remark}

Note that in both proof of Kochetkov and this manuscript, the enumeration formula does not depend on the existence criterion. So there are natural questions for further researches:  
\begin{enumerate}
    \item Can we recover the condition 
    \eqref{eq:exist_WBPTree_realW} by Theorem \ref{thm:passport_simple} or the approach here? 
    \item Can we obtain Harary-Tutte's enumeration formula \cite{HarTut64} for non-weighted trees by Theorem \ref{thm:passport_simple}?
    \item Can we derive Pakovich-Zvonkin's classification result \cite{PakZvo14} about passports with $\absbig{\Tree(\pp)} = 1$ from Theorem \ref{thm:passport_simple}?
\end{enumerate}

\bigskip
\subsection{Notations and organizations} 
More declarations about the notations we are going to adopt.

First, let us introduce a practical method for describing an WBP-tree, using the \textbf{ribbon graph} structure. We refer to \cite[\S 1.3, \S 1.5]{LsZak04}, \cite[\S 1.1]{EmM13} for details of the following definition and proposition, and \cite[\S 4.1]{KonZor03} for the ``cyclic order'' description.

\begin{definition}\label{def:cyclic_action}
    Let $T=(V,E,\wt_E)$ be a WBP-tree. Recall that $E(v)$ is the set of edges adjacent to $v$. For each $v \in V$, the orientation of the plane in which the tree is embedded induces a \textbf{cyclic action} $\sigma_v$ of order $\deg(v)$ on $E(v)$: $\forall e \in E(v)$, $\sigma_v(e)$ is the anticlockwise next edge in $E(v)$ around $v$. 
\end{definition}

\begin{proposition}\label{prop:eq_ribbon}
    An LWBP-tree is completely determined, up to isomorphisms, by an abstract tree $(V,E)$ with a family of cyclic actions $\sigma_v$ at each vertex $v$, together with the weight function $\wt_E$ and the labeling $\Lab$.
    
    More precisely, two LWBP-trees $T_1, T_2$ 
    are isomorphic if and only if there is a graph isomorphism $I:(V_1, E_1) \to (V_2, E_2)$ such that $\wt_{E_2} \circ\ I = \wt_{E_1}$, $\Lab_2 \circ\ I = \Lab_1$ and  
    \begin{equation}\label{eq:cyclic_compatible}
        I\big(\sigma_v (e) \big)= \sigma_{I(v)}\big( I(e) \big) 
    \end{equation}
    for any $v \in V_1, e \in E(v) \subset E_1$. 
\end{proposition}

\medskip
Secondly, we introduce more specific notations for the passport.
\begin{notation}[also see \cite{Zvo13}]\label{not:passport}
    For $w\in \wt(S) \subset \Rnz$, we shall present $\wt^{-1}(w) \subset S$ as a set of integers with distinct subscript $\{ w_1, w_2, \cdots, w_a \}$, where $a=\lvert \wt^{-1}(w) \rvert$. 
    And we shall denote $\pp$ in the \textbf{power notation} 
    \[ \pp = \prod_{s\in S} s^{\lambda(s)} \ . \] 
    Denote $\overline{n}:= -n \in \mathbb{Z}_{-}$ for compactness. 
    For example, $\pp=(1_1^2\ 1_2\ \ol{3}_1)$ means 
    \begin{itemize}
        \item $S=\{1_1, 1_2, \ol{3}_1\}$; 
        \item $\lambda(1_1)=2$, $\lambda(1_2) = \lambda(\ol{3}_1)=1$; 
        \item $\wt(1_1)= \wt(1_2) = 1$, $\wt(\ol{3}_1)=-3$. 
    \end{itemize} 
\end{notation}

\begin{example} \label{eg:tree_label}
    Figure \ref{fig:tree_label} exhibits some LWBP-trees of various kinds of passports. 
    \begin{enumerate}
        \item Figure \ref{subfig:tree_label_A} is a tree of trivial passport $(1^3\ \ol{3})$. 
        \item Figure \ref{subfig:tree_label_A*} is a tree of passport $(1_1^2\ 1_2\ \ol{3})$. 
        \item Figure \ref{subfig:tree_label_B}, \ref{subfig:tree_label_C},  \ref{subfig:tree_label_D} are trees of full passport $(1_1\ 1_2\ 1_3\ \ol{3})$.
        Trees in Figure \ref{subfig:tree_label_B} and \ref{subfig:tree_label_D} are isomorphic, differing by a $2\pi/3$-rotation about the central white vertex. 
        But trees in Figure \ref{subfig:tree_label_B} and \ref{subfig:tree_label_C} are not isomorphic. This can be checked by the cyclic action at the white vertex.     
        \hfill $\square$
    \end{enumerate}
\end{example}

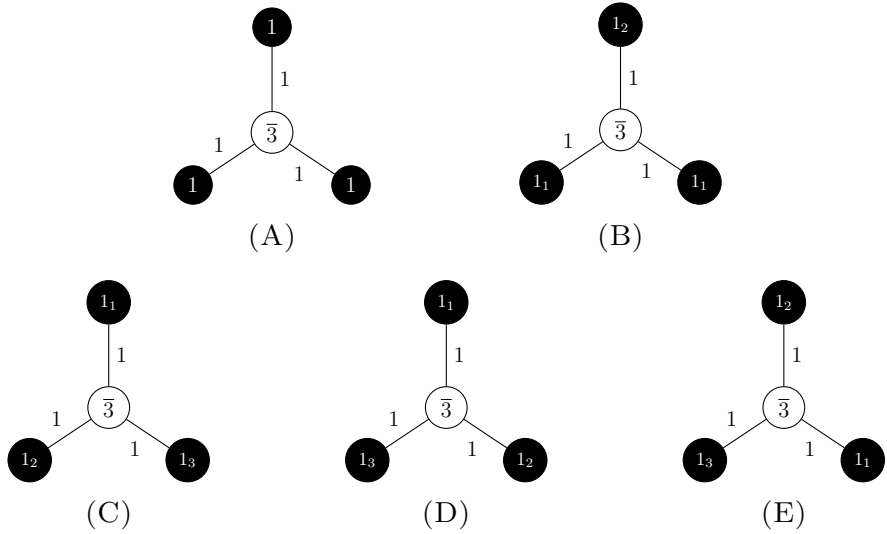
\begin{figure}[ht]
    \begin{subfigure}{0.31\textwidth}
        \centering
        \begin{tikzpicture}[scale=0.7, transform shape]
            \node[blk] (up) at (0, 2) {1};
            \node[wht] (mid) at (0, 0) {$\ol{3}$};
            \node[blk] (left) at (-1.5, -1) {1};
            \node[blk] (right) at (1.5, -1) {1};
            \draw (up) -- (mid) node[mid_auto] {1};
            \draw (left) -- (mid) node[mid_auto] {1};
            \draw (right) -- (mid) node[mid_auto] {1};
        \end{tikzpicture}
        \caption{}
        \label{subfig:tree_label_A}
    \end{subfigure}
    \begin{subfigure}{0.31\textwidth}
        \centering
        \begin{tikzpicture}[scale=0.7, transform shape]
            \node[blk] (up) at (0, 2) {\footnotesize $1_2$};
            \node[wht] (mid) at (0, 0) {$\ol{3}$};
            \node[blk] (left) at (-1.5, -1) {\footnotesize $1_1$};
            \node[blk] (right) at (1.5, -1) {\footnotesize $1_1$};
            \draw (up) -- (mid) node[mid_auto] {1};
            \draw (left) -- (mid) node[mid_auto] {1};
            \draw (right) -- (mid) node[mid_auto] {1};
        \end{tikzpicture}
        \caption{}
        \label{subfig:tree_label_A*}
    \end{subfigure}
    
    \vspace{0.3cm}
    
    \begin{subfigure}{0.3\textwidth}
        \centering
        \begin{tikzpicture}[scale=0.7, transform shape]
            \node[blk] (up) at (0, 2) {\footnotesize $1_1$};
            \node[wht] (mid) at (0, 0) {$\ol{3}$};
            \node[blk] (left) at (-1.5, -1) {\footnotesize $1_2$};
            \node[blk] (right) at (1.5, -1) {\footnotesize $1_3$};
            \draw (up) -- (mid) node[mid_auto] {1};
            \draw (left) -- (mid) node[mid_auto] {1};
            \draw (right) -- (mid) node[mid_auto] {1};
        \end{tikzpicture}
        \caption{}
        \label{subfig:tree_label_B}
    \end{subfigure}
    \begin{subfigure}{0.3\textwidth}
        \centering
        \begin{tikzpicture}[scale=0.7, transform shape]
            \node[blk] (up) at (0, 2) {\footnotesize $1_1$};
            \node[wht] (mid) at (0, 0) {$\ol{3}$};
            \node[blk] (left) at (-1.5, -1) {\footnotesize $1_3$};
            \node[blk] (right) at (1.5, -1) {\footnotesize $1_2$};
            \draw (up) -- (mid) node[mid_auto] {1};
            \draw (left) -- (mid) node[mid_auto] {1};
            \draw (right) -- (mid) node[mid_auto] {1};
        \end{tikzpicture}
        \caption{}
        \label{subfig:tree_label_C}
    \end{subfigure}
    \begin{subfigure}{0.3\textwidth}
        \centering
        \begin{tikzpicture}[scale=0.7, transform shape]
            \node[blk] (up) at (0, 2) {\footnotesize $1_2$};
            \node[wht] (mid) at (0, 0) {$\ol{3}$};
            \node[blk] (left) at (-1.5, -1) {\footnotesize $1_3$};
            \node[blk] (right) at (1.5, -1) {\footnotesize $1_1$};
            \draw (up) -- (mid) node[mid_auto] {1};
            \draw (left) -- (mid) node[mid_auto] {1};
            \draw (right) -- (mid) node[mid_auto] {1};
        \end{tikzpicture}
        \caption{}
        \label{subfig:tree_label_D}
    \end{subfigure}
    \caption{Some LWBP-trees of various kinds of passports, all with three vertices of weight $+1$ and one vertex of weight $-3$.}
    \label{fig:tree_label}
\end{figure}

\begin{remark}
Enumerating trees of a given \emph{trivial} passport is usually complicated, since such a tree may have non-trivial automorphism group. 
However, with the help of labeling, each LWBP-tree of \emph{full} passport must has trivial automorphism group. That is why we turn to full passport here. Furthermore, Theorem \ref{thm:passport_simple} implies an enumeration for general passport, where an LWBP-tree of automorphism group $G$ is counted as a rational number $1/|G|$. 
Also see \cite[\S 11.4]{APZ2020}. 
\hfill $\square$
\end{remark}

\medskip
The remaining of this manuscript is organized as follows. 
\begin{itemize}
    \item Section \ref{sec:permutation} explains the combing map. The subset $\PT(\ppf)$ of permutations that build trees through the combing map is characterized. Then the larger bijection $\TM(\ppf) \cong \PT(\ppf)$ is obtained, hence proving the non-decomposable case of Theorem \ref{thm:passport_simple}.
    \item Section \ref{sec:decomposable} proves all intermediate formulae for Theorem \ref{thm:passport_simple}. Two subsets $\PP_+(\ppf), \PT_+(\ppf)$ of permutations are introduced first, with a recursive relation between them. The smaller bijection $\TR(\ppf) \cong \PT_+(\ppf)$ is directly obtained. 
    Then the enumerations for $\PP_+(\ppf), \PT_+(\ppf)$ are given. Also see the summary at the beginning of Section \ref{sec:decomposable}. 
    \item The appendix presents some computational results. The first one lists the set $\PT_+(\ppf) \cong \TR(\ppf)$ for a concrete full passport. 
    The second one shows numerical results of Kochetkov's formula, with the total weight $4 \leq \Vert \ppf \Vert \leq 7$.  
    A notation table is attached at the end.
\end{itemize}

\bigskip\noindent
{\bf Acknowledgment. }
The authors would like to express their sincere thanks to Bin Xu for providing influential support on this work. Yi Song is supported by the National Natural Science Foundation of China (Grant No. 125B1019).

\bigskip
\section{Permutation representation} \label{sec:permutation}

In this section, we first construct the ``combing map'' (Construction \ref{cnstr:combing_step1}, \ref{cnstr:combing_step2} and Definition \ref{def:comb}) that builds a twice-marked forest from a permutation. 
Next we characterize the ``tree permutations'' (Definition \ref{def:PT}) that produces trees (rather than disconnected forests) under the combing map (Proposition \ref{prop:comb_PT_tree}). 
The reverse process from a twice-marked tree to a permutation is also constructed, called ``folding map'' (Construction \ref{cnstr:fold1}, \ref{cnstr:fold2}). 
Then we obtain the bijection from the set of tree permutations to the set of twice-marked LWBP-trees (Theorem \ref{thm:comb_biject}). This directly implies \eqref{eq:enumeration_passport_nondecomposable} of Theorem \ref{thm:passport_simple} (Corollary \ref{cor:count_nondecomp}).

Recall that for a fixed given full passport $\ppf=(S, \one, \wt)$, $N := \absbig{\ppf} = \absbig{S} \geq 2$. The index set $S$, the vertex set of a tree in $\Tree(\ppf)$ (or a forest in $\Fore(\ppf)$) are all identified with $[N]=\{1,2,\cdots,N\}$. 
Note that $\Tree(\ppf)$ might be empty while $\Fore(\ppf)$ is never empty. This is also implied by the construction next.

\bigskip
\subsection{From permutation to forest}\label{ssec:perm_to_tree}

Now we construct the combing map that eventually produces a LWBP-forest with two marked vertices from a permutation of the index set. It is inspired by the train-track technique for foliation on surfaces, especially Construction 1.7.7 of Penner-Harer \cite{PenHar92}.

We will do the following two steps to obtain a graph which turns out to be a forest. The first step is drawing a ``histogram'' region on the plane. 

\begin{construction}\label{cnstr:combing_step1}
    A closed region $R(\perm)$ in $\RR^2$ is defined as follows.
\begin{enumerate}
    \item For each $i \in [N]$, the $i$-th \textbf{cumulative sum} is $H_i := \sum_{j=1}^{i} \wt(s_j)$. \\
    Also define $H_0 := 0$. Note that $H_N = 0$. 
    \item Draw a \textbf{vertical rectangle} in $\RR^2$ (possibly a segment) for each $i$ as
    \begin{equation*}
        Rec_i :=
        \big[\ i\ ,\ i+1\ \big] \times \big[\ \min\{0, H_i\}\ ,\ \max\{0, H_i\}\ \big] \ .
    \end{equation*}
    \item Then $R(\perm) := \bigcup_{i=1}^N {Rec_i}$ is called the \textbf{region} of permutation $\perm$.
    \item The boundary of $R(\perm)$ consists of three kinds of line segments:
    \begin{itemize}
        \item the $i$-th \textbf{horizontal line} $[i, i+1] \times \{H_i\}$, 
        \item the $i$-th \textbf{rising vertical line} $\{i\} \times [H_{i-1}, H_i]$ when $\wt(s_i) > 0$, or \\
        the $i$-th \textbf{falling vertical line} $\{i\} \times [H_i, H_{i-1}]$ when $\wt(s_i) < 0$, 
        \item the segment on the x-axis: $[1, N] \times \{0\}$.
    \end{itemize}
    Notice that the length of $i$-th vertical lines is just $\absbig{H_i - H_{i-1}} = \absbig{\wt(s_i)}$. 
\end{enumerate}
\end{construction}

As an intermediate step, we need to build a labeled weighted bi-colored ribbon graph (possibly disconnected) instead of \emph{plane} graph. Soon in Proposition \ref{prop:GP_forest} we will show that it is a forest indeed, hence plane.

\begin{construction}\label{cnstr:combing_step2} 
    A weighted bi-colored ribbon graph $(V, E, \wt_E)$ with labeling $\Lab:V\to[N]$ is obtained from $R(\perm)$ as follows.
\begin{enumerate}
    \item $V$ is identified with $[N]$ and $\Lab(i):=s_i$. 
    A vertex $i$ is black or white when $\wt(s_i) > 0$ or $<0$. Hence, a black vertex $i$ is related to the $i$-th rising vertical line, while a white vertex is related to the falling one.
    \item Extend every horizontal line in both directions until it reaches a boundary vertical line of $R(\perm)$. These maximally extended segments, together with the segment on the x-axis, become the new horizontal cuts.
    \item These cuts split $R(\perm)$ into \textbf{horizontal rectangles}. Each rectangle, bounded by the $k$-th and $l$-th vertical lines, creates an edge $e$ connecting vertex $k$ to $l$ with weight $\wt_E(e)$ equal to its height; rectangle and edge share the same symbol.
    \item Take $\{e_1,\dots ,e_r\}$ as the horizontal rectangles attached to the right of $i$-th vertical line, listed bottom-to-top, and $\{e_{r+1},\dots ,e_{r+k}\}$ as those attached to the left, listed top-to-bottom. 
    Define the cyclic action $\sigma_i$ at vertex $i$ as
    \begin{equation*}
        e_1 \mapsto e_2 \mapsto \cdots e_r \mapsto e_{r+1} \mapsto \cdots e_{r+l} \mapsto e_1 \ .
    \end{equation*}
\end{enumerate}
This labeled weighted bi-colored ribbon graph is denoted by $G(\perm)$.
\end{construction}

See Figure \ref{fig:combing} as an example for $\perm=(2_2, 2_3, \ol{3}_2, \ol{3}_1, 2_1)$. 
Note that the weight at each vertex $v_i$ is exactly $\absbig{\wt(s_i)}$. 

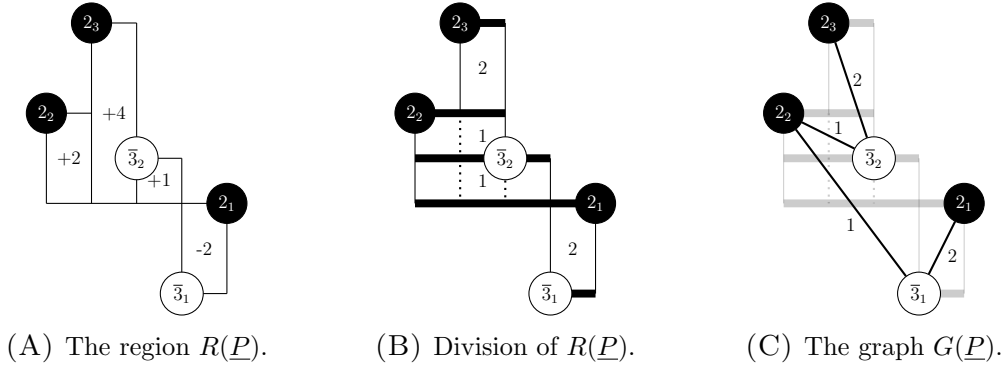
\begin{figure}[tb]
\begin{subfigure}{0.3\textwidth}    
    \centering
    \begin{tikzpicture}[scale=0.6, transform shape]
        \draw (0,0) -- (4,0); 
        \draw (0,0) -- (0,2);
        \draw (1,0) -- (1,4); 
        \draw (2,0) -- (2,4);
        \draw (3,1) -- (3,-2);
        \draw (4,-2) -- (4,0);
        \draw (0,2) -- (1,2);
        \draw (1,4) -- (2,4);
        \draw (2,1) -- (3,1);
        \draw (3,-2) -- (4,-2);
        
        \node[blk](2_2) at (0,2) {$2_2$}; 
        \node[blk](2_3) at (1,4) {$2_3$};
        \node[wht](3_2) at (2,1) {$\ol{3}_2$};
        \node[wht](3_1) at (3,-2) {$\ol{3}_1$};
        \node[blk](2_1) at (4,0) {$2_1$}; 
        \node at (0.5,1) {+2};
        \node at (1.5,2) {+4};
        \node at (2.5,0.5) {+1};
        \node at (3.5,-1) {-2};
    \end{tikzpicture}
    \caption{\footnotesize The region $R(\perm)$.} 
\end{subfigure}
\quad
\begin{subfigure}{0.3\textwidth}    
    \centering
    \begin{tikzpicture}[scale=0.6, transform shape]
        \draw (0,0) -- (0,2);
        \draw (1,2) edge[thick, dotted] (1,0) edge (1,4); 
        \draw (2,1) edge[thick, dotted] (2,0) edge (2,4);
        \draw (3,1) -- (3,-2);
        \draw (4,-2) -- (4,0);
        \draw[line width=3pt] (0,0) -- (4,0); 
        \draw[line width=3pt] (0,2) -- (2,2);
        \draw[line width=3pt] (1,4) -- (2,4);
        \draw[line width=3pt] (0,1) -- (3,1);
        \draw[line width=3pt] (3,-2) -- (4,-2);
        
        \node[blk](2_2) at (0,2) {$2_2$}; 
        \node[blk](2_3) at (1,4) {$2_3$};
        \node[wht](3_2) at (2,1) {$\ol{3}_2$};
        \node[wht](3_1) at (3,-2) {$\ol{3}_1$};
        \node[blk](2_1) at (4,0) {$2_1$}; 
        \node at (1.5, 3) {2};
        \node at (1.5, 1.5) {1};
        \node at (1.5, 0.5) {1};
        \node at (3.5, -1) {2};
    \end{tikzpicture}
    \caption{\footnotesize Division of $R(\perm)$.}
\end{subfigure}
\quad
\begin{subfigure}{0.3\textwidth}    
    \centering
    \begin{tikzpicture}[scale=0.6, transform shape]
        \draw[opacity=0.2] (0,0) -- (0,2);
        \draw[opacity=0.2] (1,2) edge[thick, dotted] (1,0) edge (1,4); 
        \draw[opacity=0.2] (2,1) edge[thick, dotted] (2,0) edge (2,4);
        \draw[opacity=0.2] (3,1) -- (3,-2);
        \draw[opacity=0.2] (4,-2) -- (4,0);
        \draw[opacity=0.2, line width=3pt] (0,0) -- (4,0); 
        \draw[opacity=0.2, line width=3pt] (0,2) -- (2,2);
        \draw[opacity=0.2, line width=3pt] (1,4) -- (2,4);
        \draw[opacity=0.2, line width=3pt] (0,1) -- (3,1);
        \draw[opacity=0.2, line width=3pt] (3,-2) -- (4,-2);
        
        \node[blk](2_2) at (0,2) {$2_2$}; 
        \node[blk](2_3) at (1,4) {$2_3$};
        \node[wht](3_2) at (2,1) {$\ol{3}_2$};
        \node[wht](3_1) at (3,-2) {$\ol{3}_1$};
        \node[blk](2_1) at (4,0) {$2_1$}; 
        \draw[thick] (2_3) -- node[pos=0.7, above=10pt] {2} (3_2);
        \draw[thick] (2_2) -- node[pos=0.7, above] {1} (3_2);
        \draw[thick] (2_2) -- node[below=6pt] {1} (3_1);
        \draw[thick] (2_1) -- node[pos=0.1, below=10pt] {2} (3_1);
        
    \end{tikzpicture}
    \caption{\footnotesize The graph $G(\perm)$.}
\end{subfigure}
    \caption{The combing map.}
    \label{fig:combing}
\end{figure}

\medskip
Now we show $G(\perm)$ is a forest. We start with basic and important properties. 

\begin{lemma}\label{lem:edge_above_below}
    Each horizontal rectangle in $R(\perm)$ defines an edge $e$ connecting some black vertex $k$ and some white vertex $l$. The corresponding rectangle is above the x-axis when $k < l$, and below the x-axis when $k > l$.
\end{lemma}
\begin{proof}
    Each horizontal rectangle ends at a rising vertical lines and a falling one, thus the edge $e$ connects a black vertex and a white one.
    For horizontal rectangle above the x-axis, the rising vertical line is at the left of the falling one, thus $s_k$ appears earlier in $\perm$ and $k < l$. 
\end{proof}

\medskip
\begin{proposition}\label{prop:edge_criterion}
    For $1 \leq k < l \leq N$, there is an edge connecting vertices $k, l$ if and only if 
        \[ \min\left\{ H_k, \cdots, H_{l-1} \right\} > \max \left\{ H_{k-1}, H_{l}, 0 \right\} \  \]
    or
        \[ \max\left\{ H_k, \cdots, H_{l-1} \right\} < \min \left\{ H_{k-1}, H_{l}, 0 \right\} \ . \]
    Moreover, the edge lies above or below the x-axis according to whether the first or second condition holds, and in either case its weight is the absolute value of the difference between the two sides of the inequality.
\end{proposition}
\begin{proof}
    We prove only the case where vertex $k$ is black. 
    If there is an edge connecting $k, l$, then by Lemma \ref{lem:edge_above_below}, there exists a horizontal rectangle $e$ attached to the right of $k$-th rising vertical line and the left of $l$-th falling vertical line. 
    By step (2) of Construction \ref{cnstr:combing_step2}, the upper horizontal edge of $e$ must come from the lowest horizontal line of $Rec_k, \cdots, Rec_{l-1}$, whose height is exactly $\min\{H_j\}_{j=k}^{l-1}$. The lower horizontal edge of $e$ must come from the highest one among the $(k-1)$-th and $l$-th horizontal line and the x-axis. This height is just $\max \left\{ H_{k-1}, H_{l}, 0 \right\}$. By step (3) of Construction \ref{cnstr:combing_step2}, the weight on this edge is the difference of these two height. 
    The converse is also proved by such discussion on the upper and lower horizontal edge of a rectangle. 
\end{proof}

\begin{corollary}\label{cor:rec_length_strict}
    For all horizontal rectangles attached to the right (resp.\ left) of the $i$-th vertical line, their lengths strictly decrease (resp.\ increase) with increasing vertical position, and vise versa.
\end{corollary}
\begin{proof}
    We prove only the case where $\wt(s_i)>0$. Assume that $e_1, e_2$ are two horizontal rectangles attached to the right of $i$-th vertical line, ending at $j$-th and $k$-th vertical line respectively. By Lemma \ref{lem:edge_above_below} they lie above the x-axis. 
    From the proof of Proposition \ref{prop:edge_criterion}, the height of lower horizontal edge of $e_1$ is $\max\{H_{i-1}, H_j, 0\}$, and the height of upper edge of $e_2$ is $\min\{H_i, \cdots H_{k-1}\}$. 
    
    If $e_1$ is higher than $e_2$, then we have
    \[ \min\{H_i,\cdots H_{k-1}\} \leq \max\{H_{i-1},H_j,0\} < \min\{H_i,\cdots H_{j-1}\} . \]
    Hence $j<k$. 
    Conversely, if $j<k$, then $\min\{H_i,\cdots H_{k-1}\} \leq \min\{H_i,\cdots H_j\}$. The existence of $e_1$ implies $\min\{H_i,\cdots H_{j-1}, H_j\} = H_j$, by Proposition \ref{prop:edge_criterion}. 
    Then 
    \[ \max\{H_{i-1}, H_j, 0\} \geq H_j = \min\{H_i,\cdots H_j\} \geq \min\{H_i, \cdots H_{k-1}\}.  \]
    
    For rectangles attached to the left, proofs are similar.
\end{proof}

\medskip
\begin{corollary}\label{cor:rec_separating}
    If there is an edge $e$ between vertices $k, l$ with $1 \leq k < l \leq N$, then there is no edge connecting $\{{k+1}, \cdots, {l-1}\}$ to $V\setminus\{k,\cdots,l\}$.  
\end{corollary}
\begin{proof}
    We prove only the case where $\wt(s_k)>0$. By Proposition \ref{prop:edge_criterion},
    \[ \min\left\{ H_k, \cdots, H_{l-1} \right\} > \max \left\{ H_{k-1}, H_{l}, 0 \right\} .  \]
    Then $H_k, H_{l-1}>0$. For any $i\in\{k+1,\cdots,l-1\}$, we have 
    \[ H_{k-1} < H_i ,\ H_l < H_{i-1} . \]

    For any $j\in\{1,\cdots,k-1\}$, $j \leq k-1 < k \leq i-1$. Then 
    \[ \min\{H_j, \cdots, H_{i-1} \} \leq H_{k-1} < H_i \leq \max\{H_{j-1},H_i,0\}  \]
    and 
    \[ \max\{H_j, \cdots, H_{i-1} \} \geq H_{k} > 0 \geq \min\{H_{j-1},H_i,0\} .\]
    For $j\in\{l+1,\cdots,N\}$, the proof is similar.
\end{proof}

\medskip
\begin{proposition}\label{prop:GP_forest}
    $(V, E, W_E)$ is a simple graph without closed path. Hence $G(\perm)$ is an LWBP-forest in $\Fore(\ppf)$.
\end{proposition}
\begin{proof}
    $G(\perm)$ contains no loop at any vertex by definition. It contains no multi-edge because two horizontal rectangles attached to the same vertical line on the same side must have different length, by Corollary \ref{cor:rec_length_strict}. 
    
    Assume that $G(\perm)$ contains a closed path which visits each vertex at most twice. Let $i$ be the vertex with smallest index on this closed path. 
    
    First assume $i$ is black. Among two horizontal rectangles attached to the right of $i$-th vertical line and corresponding to the two edges on the path, let $e_0$ be the one whose horizontal position is higher. Let $j$ be the other vertex of $e_0$. By Lemma \ref{lem:edge_above_below}, $j>i$.  
    
    Consider the two horizontal rectangles $e_{-1}, e_{+1}$ corresponding to the adjacent edges of $e_0$ in the closed path. Let $i,k$ be the vertices of $e_{-1}$, and $j, l$ be the vertices of $e_{+1}$. Then $j, k$ are white and $l$ is black. By Lemma \ref{lem:edge_above_below}, both $e_0, e_{-1}$ lie above the x-axis. Due to the choice of $e_0$, we see that $e_{-1}$ is also attached to the right of the $i$-th vertical line and lower than $e_0$. By Corollary \ref{cor:rec_length_strict}, $e_{-1}$ is strictly longer than $e_0$, hence $k>j$. Then applying Proposition \ref{prop:edge_criterion} to $e_0, e_{-1}$, we see $H_{j-1} > H_{j} > H_{i-1}$. 
    
    Now consider vertex $l$ of $e_{+1}$. 
    Since $i$ is assumed to be the smallest index on this closed path and is already visited twice, $l > i$. The horizontal rectangle $e_{-1}$ lies above x-axis, so $e_{+1}$ is also above x-axis and $l < j$ by Lemma \ref{lem:edge_above_below}. 
    Therefore, $i<l<j<k$ (see Figure \ref{fig:no_closed_path}).   
    Deleting edges $e_{-1}, e_0, e_{+1}$ in the closed path, one gets a path from $l$ to $k$ without passing $i,j$. 
    However, applying Corollary \ref{cor:rec_separating} to $e_0$, there is no edge from $\{{i+1}, \cdots, {j-1}\}$ to $V \setminus \{i, \cdots, j\}$. Since $i<l<j<k$, such path from $l$ to $k$ can not exist. 

    If $i$ is white, let $e_0$ be the one whose horizontal position is lower. The arguments are similar. 
\end{proof}

\begin{figure}[ht]\centering
\begin{subfigure}{0.45\textwidth} \centering
    \begin{tikzpicture}[scale=0.5, transform shape, every node/.style={font=\LARGE}]
        \draw[line width=2pt, dashed] (-1.3,0) -- (10,0);
        \filldraw[opacity=0.6, lightgray] (0,1) rectangle (9,2);
        \filldraw[opacity=0.6, lightgray] (0,3) rectangle (6,4);
        \filldraw[opacity=0.6, lightgray] (4,5) rectangle (6,6);
        \draw (0,1) edge (0,4) edge[thick, dotted] (0,0);
        \draw (4,5) edge (4,6) edge[thick, dotted] (4,0); 
        \draw (6,3) edge (6,6) edge[thick, dotted] (6,0);
        \draw (9,1) edge (9,2) edge[thick, dotted] (9,0);
        \draw[line width=3pt] (0,1) -- (9,1); 
        \draw[line width=3pt] (0,2) -- (9,2);
        \draw[line width=3pt] (0,3) -- (6,3);
        \draw[line width=3pt] (0,4) -- (6,4);
        \draw[line width=3pt] (4,5) -- (6,5);
        \draw[line width=3pt] (4,6) -- (6,6);    
        \node at (3, 3.5) {$e_0$};
        \node at (5, 1.5) {$e_{-1}$};
        \node at (5, 5.5) {$e_{+1}$};
        \node at (0, -0.5) {$i$};
        \node at (6, -0.5) {$j$};
        \node at (9, -0.5) {$k$};
        \node at (4, -0.5) {$l$};
    \end{tikzpicture}
\end{subfigure}
\hfill
\begin{subfigure}{0.45\textwidth} \centering
    \begin{tikzpicture}[scale=0.53, transform shape, every node/.style={font=\LARGE}]
        \draw[thick] (0, 2) -- node[pos=0.5, above=2pt] {$e_0$} (6, 3);
        \draw[thick] (0, 2) -- node[pos=0.6, above=2pt] {$e_{-1}$} (9, 0);
        \draw[thick] (4, 5) -- node[pos=0.7, above=6pt] {$e_{+1}$} (6, 3);
        \node[blk]() at (0,2) {$v_i$}; 
        \node[blk]() at (4,5) {$v_l$}; 
        \node[wht]() at (6,3) {$v_j$}; 
        \node[wht]() at (9,0) {$v_k$}; 
    \end{tikzpicture}
    \vspace{17pt}
\end{subfigure}
\caption{Three consecutive edges on the closed path.}
\label{fig:no_closed_path}
\end{figure}
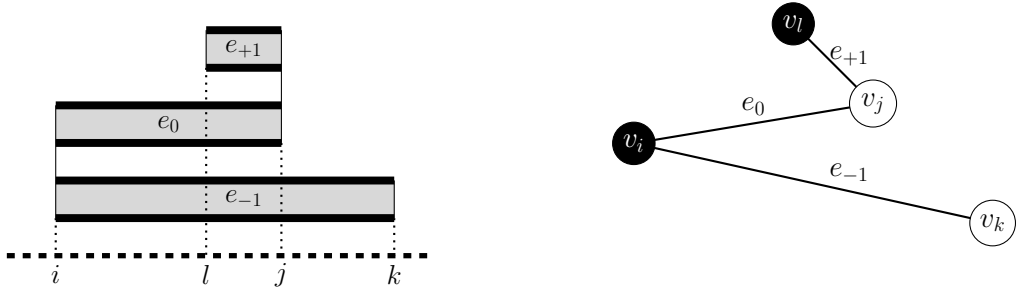

\begin{definition}\label{def:comb}
    Let $\FM(\ppf), \TM(\ppf)$ be the set of twice-marked forests and trees of passport $\ppf$ in Definition \ref{def:twice_marked_tree}. Then the mapping 
    \begin{align}\label{eq:combing}\begin{split}
        \comb: \quad \PP(\ppf) \quad & \longrightarrow \qquad \FF M(\ppf) \\
        \perm=(s_i)_{i=1}^N &\longmapsto \big( G(\perm); s_1, s_N \big)
    \end{split}\end{align}
    is called the \textbf{combing map}, where $G(\perm)$ is defined by Construction \ref{cnstr:combing_step1} and \ref{cnstr:combing_step2}. 
\end{definition}

\bigskip
\subsection{Tree permutation}\label{ssec:tree_perm}
Those $\perm\in\PP(\ppf)$ with connected $G(\perm)$ are fully characterized by the following conditions. The structure of general disconnected $G(\perm)$ will be partially given in Section \ref{ssec:pos_perm} later. 

\begin{proposition}\label{prop:comb_PT_tree}
    $G(\perm)$ is a tree if and only if all the followings hold.
    \begin{enumerate}
        \item $H_i \neq 0$ for all $i \in [N-1]$.
        \item If $H_{k-1}=H_l>0$ for some $k<l\in[N]$, then 
        \[ \min\{H_k,\cdots H_{l-1}\} < H_{k-1} = H_l . \]
        \item If $H_{k-1}=H_l<0$ for some $k<l\in[N]$, then 
        \[ \max\{H_k,\cdots H_{l-1}\} > H_{k-1} = H_l . \]
    \end{enumerate}
\end{proposition}
\begin{proof}
    Obviously, if there exists some $i \in [N-1]$ with $H_i=0$, then the region $R(\perm)$ is already disconnected. And there is no edge from $\{1,\cdots, i\}$ to $\{i+1,\cdots,N\}$. We therefore assume that (1) holds. 

    \medskip
    First assume that condition (2) does not hold. Let $k<l\in[N]$ such that 
    \begin{equation}\label{eq:horizontal_pair}
    \min\{H_k,\cdots H_{l-1}\} \geq H_{k-1} = H_l>0.
    \end{equation}
    If there exists some $l' \in \{k, \cdots, l-1\}$ with $H_{l'} = H_{k-1} = H_{l}$, then $l'>k$ since $\wt(s_k)>0$. One can replace $l$ by $l'$ and \eqref{eq:horizontal_pair} still holds. So we can assume $H_k,\cdots, H_{l-1} > H_{k-1}$ in \eqref{eq:horizontal_pair}. By Proposition \ref{prop:edge_criterion}, there is an edge from vertex $k$ to $l$.
    Since $H_{k} > H_{k-1} > 0$, for any $j \in \{1,\cdots,k-1\}$, 
    \begin{align*}
    \min\{H_j,\cdots H_{k-1}\} & \leq H_{k-1} < H_k \leq \max\{H_k, H_{j-1},0\},\ \\ 
    \max\{H_j,\cdots H_{k-1}\} & \geq H_{k-1} > \ 0\ \geq \min\{H_k, H_{j-1},0\} . 
    \end{align*}
    So there is no edge from $\{1,\cdots,k-1\}$ to $k$. 
    By similar arguments, together with Corollary \ref{cor:rec_separating}, there is no edge from $\{k,\cdots,l\}$ to the other vertices. So $G(\perm)$ is disconnected.
    Condition (3) can be check by the same steps. 

    \medskip
    Conversely, if $G(\perm)$ is disconnected while (1) holds, let $i$ be the first vertex not in the component containing $1$ ($i>1$). In other words, there is no edge from $\{1,\cdots, i-1\}$ to $i$. 
    
    When $i$ is black, $H_i>H_{i-1}$. 
    We claim that $H_{i-1}>0$. Otherwise, since $H_0=0$ and $H_{i-1}<0$, there exists $1\leq j \leq i-1$ such that
    $\max\{H_j,\cdots H_{i-1}\} = H_{i-1}$, while $H_{j-1} > H_{i-1}$. 
    Then $\max\{H_j,\cdots H_{i-1}\} = H_{i-1} < \min\{H_{j-1},H_i,0\}$ and there is an edge from $j$ to $i$, contradict to our assumption of $i$. 

    Next we claim that (2) is not satisfied.
    Since $H_0=0$ and $H_{i-1}>0$, there exist $1\leq j \leq i-1$ such that
    $\min\{H_j,\cdots H_{i-1}\} = H_{i-1}$, while $H_{j-1} < H_{i-1}$.
    Since $H_{N}=0$ and $H_i > H_{i-1} >0$, there exist $k\geq i+1$ such that 
    $H_{k} \leq H_{i-1}$ and $\min\{H_i,\cdots H_{k-1}\} > H_{i-1}$. 
    If $H_k < H_{i-1}$, then 
    \[ \min\{H_j,\cdots H_{i-1}, H_i, \cdots H_{k-1}\} = H_{i-1} > \max\{H_{j-1}, H_k,0\} \]
    and 
    \[ \min\{H_i, \cdots H_{k-1}\} > H_{i-1} \geq \max\{H_{i-1}, H_k,0\} .\]
    There are edges between $j,k$ and $i,k$, which contracts to our assumption of $i$ again. So $H_k = H_{i-1}>0$ and $\min\{H_i,\cdots H_{k-1}\} > H_{i-1}$, by the choice of $k$.

    When $i$ is white, similar arguments show that (3) is not satisfied.
\end{proof}

\begin{definition}\label{def:PT}
    $\perm\in\PP(\ppf)$ is called a \textbf{tree permutation} if it satisfies all three conditions in Proposition \ref{prop:comb_PT_tree}.
    The subset of all tree permutations is denoted by $\PT(\ppf)$. Then 
    \[ \comb\big(\PT(\ppf)\big) \subset \TM(\ppf) \subset \FM(\ppf) \ . \]
\end{definition}

\bigskip
\subsection{From tree to permutation}\label{ssec:tree_to_perm}
The combing map is not injective in general, but turns out to be injective when restricted on $\PT(\ppf)$. Its inverse process, called the folding map, is constructed. This leads to our first important bijection.

\begin{theorem}\label{thm:comb_biject}
    The restricted map $\comb : \PT(\ppf) \to \TM(\ppf)$ is a bijection. 
\end{theorem}

\begin{corollary}\label{cor:count_nondecomp}
    Theorem \ref{thm:passport_simple} holds for non-decomposable full passport.  
\end{corollary}

\begin{proof}[Proof of Corollary]
    For any permutation $\perm \in \PP(\ppf)$ and $k < l \in [N]$, $H_k \neq H_l$ since $\ppf$ is non-decomposable. In particular, for all $i \in [N-1]$, $H_i \neq 0$. Hence $\perm \in \PT(\ppf)$, and then $\absbig{\PT(\ppf)} = \absbig{\PP(\ppf)} = N!$. 
    By Theorem \ref{thm:comb_biject}, $\absbig{\PT(\ppf)} = \absbig{\TM(\ppf)} = \absbig{\Tree(\ppf)} \cdot N(N-1)$, so $\absbig{\Tree(\ppf)} = (N-2)!$.
\end{proof}

\begin{remark}
    It can be checked that the combing map is always surjective. This requires the structure of a general permutation, which is sketched in Remark \ref{rmk:horizontal_division_general}. Nevertheless, these facts are not used in the proof of main theorem. 
    \hfill $\square$
\end{remark}

To find the inverse map, we represent a twice-marked LWBP-tree by a region consisted of horizontal rectangles. Then cutting this region vertically will recover a permutation. Here is the key observation.

\begin{proposition}\label{prop:sign-changing_point}
    For $\perm \in \PT(\ppf)$, a \textbf{sign-changing point} of $\perm$ is a number $i \in [N]$ that satisfies $H_{i-1} \cdot H_i < 0$. 
    Assume the sign-changing points of $\perm$ are $ 1 < \ell(1) < \cdots < \ell(M) < N$. 
    Then the unique path in $G(\perm)$ from $1$ to $N$ is just 
    \[ \ell:= \big(\ 1={\ell(0)},\ {\ell(1)},\ \cdots,\ {\ell(M)},\ {\ell(M+1)=N}\ \big) \ . \]
\end{proposition}

\begin{proof}
Since $G(\perm)$ is a tree, we only need to show that there is an edge between $\ell(i)$ and $\ell(i+1)$, $\forall\ 0 \leq i \leq M$. 
Because $\perm \in \PT(\pp_F)$, $H_j \neq 0$ except for $j=0=\ell(0)-1$ and $j=N=\ell(M+1)$.

For $1 \leq i \leq M-1$ with $H_{\ell(i)-1} < 0$ and $H_{\ell(i)} > 0$, there is no sign-changing point between $\ell(i)$ and $\ell(i+1)$, then $H_j > 0$ for all $\ell(i) \leq j \leq \ell(i+1)-1$, and $H_{\ell(i+1)} \leq 0$.
Since $\min\{H_{\ell(i)} , \cdots , H_{\ell(i+1)-1}\} > 0 = \max\{H_{\ell(i)-1}, H_{\ell(i+1)}, 0\} $,
there must be an edge between $\ell(i)$ and $\ell(i+1)$ by Proposition \ref{prop:edge_criterion}. 
For $1 \leq i \leq M-1$ with $H_{\ell(i)-1} > 0$ and $H_{\ell(i)} < 0$, the arguments are similar. 
For $i = 0$ or $M$, one of $H_{\ell(i)-1}, H_{\ell(i+1)}$ is zero and the above arguments hold again. 
\end{proof}

\medskip
Now we are ready to construct the inverse map. Some definition is needed. 
Let $(T;a,b) \in \TM(\ppf)$ and $\ell=(a=\ell(0), \ell(1),\cdots,\ell(M+1)=b)$ be the unique path from vertex $a$ to $b$ in $T$. The edge from $\ell(i)$ to $\ell(i+1)$ in $\ell$ is denoted by $e_{i}^{0}$, $0\leq i \leq M$.  

\begin{construction}\label{cnstr:fold1}
    For each $(T;a,b) \in \TM(\ppf)$, we shall define a closed region $\widehat{R}(T;a,b)$ in $\RR^2$ that consists of horizontal rectangles, with each rectangle $\mathbf{e}$ corresponding to an edge $e$ of $T$. See Figure \ref{fig:folding} as a quick preview.
    \begin{enumerate}
        \item For $0 \leq i \leq M$, define the horizontal rectangle
        \begin{equation*} \mathbf{e_i^{0}} :=
        \begin{cases}
            [i, i+1] \times [0,\quad \wt_E(e_i^{0})], &\textrm{ if } \ell(i) \textrm{ is black; }\\ 
            [i, i+1] \times [-\wt_E(e_i^{0}), 0], &\textrm{ if } \ell(i) \textrm{ is white. }\\ 
        \end{cases}
        \end{equation*}
        \item For $0\leq k \leq M+1$, label edges in $E(\ell(k))$ as 
        \[ \big( e_{k}^{0}, e_{k}^{1},\cdots e_{k}^{\alpha}, e_{k-1}^{0}, e_{k-1}^{1},\cdots e_{k-1}^{\beta} \big) \]
        anticlockwise if $\ell(k)$ is black, and clockwise if $\ell(k)$ is white. 
        In addition, if $k=0$, $e_{k-1}^0, e_{k-1}^{1},\cdots,e_{k}^{\beta}$ do not appear; if $k=M+1$, $e_{k}^0,\cdots, e_{k}^\alpha$ do not appear.
        \begin{itemize}
        \item When $k$ is black, draw $\alpha$ rectangles left-aligned above rectangle $\mathbf{e_{k}^{0}}$ :
        \[ \mathbf{e_k^i} := [\ k,\ k+3^{-i}\ ] \times [\ \textstyle\sum_{j=0}^{i-1}\wt_E({e_k^j}) ,\ \textstyle\sum_{j=0}^{i}\wt_E({e_k^j})\ ],\ 1 \leq i \leq \alpha; \]
        and $\beta$ rectangles right-aligned below $\mathbf{e_{k}^{0}}$: 
        \[ \mathbf{e_{k-1}^i} := [\ k-3^{-i},\ k\ ] \times [\ -\textstyle\sum_{j=0}^{i}\wt_E({e_k^j}) ,\ -\textstyle\sum_{j=0}^{i-1}\wt_E({e_k^j})\ ],\ 1 \leq i \leq \beta.\] 
        \item When $k$ is white, draw $\alpha$ rectangles left-aligned below $\mathbf{e_{k}^{0}}$ : 
        \[ \mathbf{e_k^i} := [\ k,\ k+3^{-i}\ ] \times [\ -\textstyle\sum_{j=0}^{i}\wt_E({e_k^j}) ,\ -\textstyle\sum_{j=0}^{i-1}\wt_E({e_k^j})\ ],\ 1 \leq i \leq \alpha;\]
        and $\beta$ rectangles right-aligned above $\mathbf{e_{k}^{0}}$: 
        \[ \mathbf{e_{k-1}^i} := [\ k-3^{-i},\ k\ ] \times [\ \textstyle\sum_{j=0}^{i-1}\wt_E({e_k^j}) ,\ \textstyle\sum_{j=0}^{i}\wt_E({e_k^j})\ ],\ 1 \leq i \leq \beta.\] 
        \end{itemize}

        \item Repeat the following until every edge has its corresponding horizontal rectangle drawn. 
        Assume that $e^0$ is an edge connecting $v,w\in V$ and every edge in $E(v)$ has its rectangle drawn, while the only such edge in $E(w)$ is $e^0$. Label edges in $E(w)$ as
        $ (e^0, e^1, \cdots, e^\gamma) $
        anticlockwise if the rectangle $\mathbf{e^0}$ is above the x-axis, and clockwise if $\mathbf{e^0}$ is below the x-axis. 
        \begin{itemize}
        \item When $\mathbf{e^0}$ is above the x-axis and $w$ is black, let $\mathbf{e^0}=[L, L+l]\times[H,H+\wt_E(e^0)]$. Draw $\alpha$ rectangles left-aligned above $\mathbf{e^0}$:
        \[ \mathbf{e^i} := [\ L,\ L+3^{-i}l\ ] \times [\ H+\textstyle\sum_{j=0}^{i-1}\wt_E({e_k^j}) ,\ H+\textstyle\sum_{j=0}^{i}\wt_E({e_k^j})\ ] . \]
        \item When $\mathbf{e^0}$ is above the x-axis and $w$ is white, let $\mathbf{e^0}=[L-l, L]\times[H,H+\wt_E(e^0)]$. Draw $\alpha$ rectangles right-aligned above $\mathbf{e^0}$:
        \[ \mathbf{e^i} := [\ L-3^{-i}l,\ L\ ] \times [\ H+\textstyle\sum_{j=0}^{i-1}\wt_E({e_k^j}) ,\ H+\textstyle\sum_{j=0}^{i}\wt_E({e_k^j})\ ] . \]
        \item When $\mathbf{e^0}$ is below the x-axis and $w$ is black, let $\mathbf{e^0}=[L-l, L]\times[H-\wt_E(e^0),H]$. Draw $\alpha$ rectangles right-aligned below $\mathbf{e^0}$:
        \[ \mathbf{e^i} := [\ L-3^{-i}l,\ L\ ] \times [\ H-\textstyle\sum_{j=0}^{i}\wt_E({e_k^j}) ,\ H-\textstyle\sum_{j=0}^{i-1}\wt_E({e_k^j})\ ] . \]
        \item When $\mathbf{e^0}$ is below the x-axis and $w$ is white, let $\mathbf{e^0}=[L, L+l]\times[H-\wt_E(e^0),H]$. Draw $\alpha$ rectangles left-aligned below $\mathbf{e^0}$:
        \[ \mathbf{e^i} := [\ L,\ L+3^{-i}l\ ] \times [\ H-\textstyle\sum_{j=0}^{i}\wt_E({e_k^j}) ,\ H-\textstyle\sum_{j=0}^{i-1}\wt_E({e_k^j})\ ] . \]
        \end{itemize}

        \item All horizontal rectangles have pairwise disjoint interiors. The union of these rectangles is denoted by $\widehat{R}(T;a,b)$. 
    \end{enumerate}
\end{construction}

\begin{figure}[t]
    \centering
\makebox[\textwidth]{
\begin{tikzpicture}[scale=0.6, transform shape]
    \draw (-3,-1) rectangle (0 ,0 ) ;
    \draw (0 ,0 ) rectangle (3 ,1 ) ;
    \node at (-1.5,-0.5) {\large$\mathbf{e^0_1}$};
    \node at ( 1.5, 0.5) {\large$\mathbf{e^0_2}$};

    \node[blk](11) at (0 ,-5) {$3_1$};
    \node[wht](12) at (-3,-5) {$\bar{3}_2$};
    \node[wht](13) at (3 ,-5) {$\bar{3}_3$};
    \draw[line width=1.4pt, double=white, double distance=1pt, black] (12)--(11)--(13);

    \node at (0 ,-7.5) {\Large $(\bar{3}_2\quad,\quad 3_1\quad,\quad \bar{3}_3)$};

    \begin{scope}[shift={(8.5,0)}]
    \draw (-3,-1) rectangle ( 0, 0) ;
    \draw ( 0, 0) rectangle ( 3, 1) ;
    \draw (-3,-3) rectangle (-2,-1) ; 
    \draw ( 0, 1) rectangle ( 1, 2) ;
    \draw ( 2, 1) rectangle ( 3, 3) ;

    \node[blk](21) at (0 ,-5) {$3_1$};
    \node[wht](22) at (-3,-5) {$\bar{3}_2$};
    \node[wht](23) at (3 ,-5) {$\bar{3}_3$};
    \draw[line width=1.4pt, double=white, double distance=1pt, black] (22)--(21)--(23);
    \node[blk](24) at (-2,-6) {$2_1$};
    \node[wht](25) at ( 1,-4) {$\bar{3}_1$};
    \node[blk](26) at ( 2,-4) {$2_3$};
    \draw[thick] (22) to (24);
    \draw[thick] (21) to (25);
    \draw[thick] (23) to (26);

    \node at (0,-7.5) {\Large $(\textcolor{lightgray}{\bar{3}_2\ ,\ } 2_1\ ,\ \textcolor{lightgray}{3_1\ ,\ }\bar{3}_1\ ,\ 2_3\ ,\  \textcolor{lightgray}{\bar{3}_3})$};
    \end{scope}

    \begin{scope}[shift={(17,0)}]
    \draw (-3,-1) rectangle ( 0, 0) ;
    \draw ( 0, 0) rectangle ( 3, 1) ;
    \draw (-3,-3) rectangle (-2,-1) ; 
    \draw ( 0, 1) rectangle ( 1, 2) ;
    \draw ( 2, 1) rectangle ( 3, 3) ;
    \draw ( 1, 2) rectangle (2/3,4) ;
    \draw[line width=3pt] (-3,-3) -- (-3, 0); 
    \draw[line width=3pt] (-2,-3) -- (-2,-1);
    \draw[line width=1pt, dotted] (-2,-1) -- (-2, 0); 
    \draw[line width=3pt] ( 0,-1) -- ( 0, 2); 
    \draw[line width=3pt] (2/3,2) -- (2/3,4);
    \draw[line width=1pt, dotted] (2/3,2) -- (2/3, 0);
    \draw[line width=3pt] ( 1, 1) -- ( 1, 4);
    \draw[line width=1pt, dotted] (1,1) -- (1,0);
    \draw[line width=3pt] ( 2, 1) -- ( 2, 3);
    \draw[line width=1pt, dotted] (2,1) -- (2,0);
    \draw[line width=3pt] ( 3, 0) -- ( 3, 3); 

    \node[blk](21) at (0 ,-5) {$3_1$};
    \node[wht](22) at (-3,-5) {$\bar{3}_2$};
    \node[wht](23) at (4 ,-5) {$\bar{3}_3$};
    \draw[line width=4pt] (22)--(21)--(23);
    \draw[line width=1.4pt, white] (22)--(21)--(23);
    \node[blk](24) at (-2,-6) {$2_1$};
    \node[wht](25) at (1.7,-4) {$\bar{3}_1$};
    \node[blk](26) at ( 3,-4) {$2_3$};
    \draw[thick] (22) to (24);
    \draw[thick] (21) to (25);
    \draw[thick] (23) to (26);
    \node[blk](27) at (1,-2.7) {$2_2$};
    \draw[thick] (25) to (27);
    \node at (0.5,-7.5) {\Large $\perm(T;a,b) = (\textcolor{lightgray}{\bar{3}_2, 2_1, 3_1,} 2_2, \textcolor{lightgray}{\bar{3}_1, 2_3, \bar{3}_3})$};
    \end{scope}

    \begin{scope}[shift={(8.5,6)}]
    \node[blk](v2) at ( 0, 0) {$3_1$};
    \node[wht](a) at (-3, 1) {$\bar{3}_2$};
    \node[wht](b) at ( 4, 1) {$\bar{3}_3$};
    \draw[thick] (v2) to (a);
    \draw[thick] (v2) to (b);
    \node[blk](v1) at (-1, 2) {$2_1$};
    \node[wht](v3) at ( 1, 2) {$\bar{3}_1$};
    \node[blk](v5) at ( 3,-1) {$2_3$};
    \draw[thick] (a) to (v1);
    \draw[thick] (v2) to (v3);
    \draw[thick] (v5) to (b);
    \node[blk](v4) at ( 3, 2) {$2_2$};
    \draw[thick] (v4) to (v3);
    \node at (-4,1) {\Large $a=$};
    \node at (5,1) {\Large $=b$};
    \draw[line width=1.4pt, double=white, double distance=1pt, black] (a)--node[pos=0.6,auto=right]{\Large $e^0_1$} (v2)--node[pos=0.4,auto=right]{\Large $e^0_2$} (b);
    \end{scope}
\end{tikzpicture}
}
    \caption{The folding map.}
    \label{fig:folding}
\end{figure}
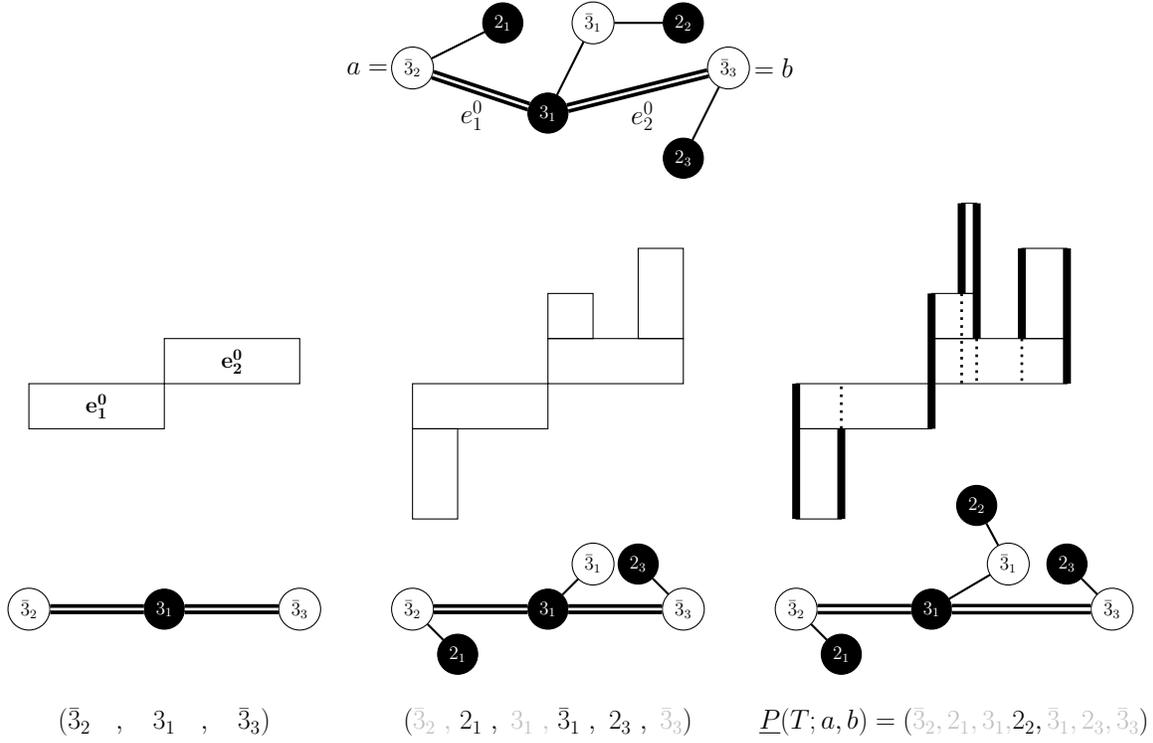

The following lemma is immediate from the construction, so proof is omitted. 
\begin{lemma}\label{lem:for_fold}
    For any $v\in V$, the horizontal rectangles corresponding to edges in $E(v)$ all terminate on a common boundary vertical segment of $\widehat{R}(T;a,b)$, whose length is exactly $\absbig{\wt_V(v)}$; these segments for distinct vertices have distinct x-coordinates. 
    Moreover, the vertical segments associated with marked vertices $a$ and $b$ are precisely those of minimal and maximal x-coordinate, respectively. 
    \hfill $\square$
\end{lemma}

\newcommand{\fold}{\mathrm{Fold}} 
\begin{construction}\label{cnstr:fold2}
    There is a bijection between the boundary vertical segments of $\widehat{R}(T;a,b)$ and the vertex set $V$. Listing these segments in increasing order of x-coordinate gives a sequence $\perm(T;a,b) \in \PP$. This is called the \textbf{folding map} 
    \begin{align}\label{eq:folding}\begin{split}
        \fold: \ \TM(\ppf) \ & \longrightarrow \ \PP(\ppf) \\
        \ (T;a,b) \ & \longmapsto \perm(T;a,b)
    \end{split} \quad . \end{align}
\end{construction}

\medskip
\begin{proof}[Proof of Theorem \ref{thm:comb_biject}] 
    This is done by directly verifying that $\comb$ and $\fold$ are mutually inverse maps.

    \medskip
    (1). $\comb(\fold(T;a,b)) = (T;a,b)$ for all $(T;a,b)\in\TM(\ppf)$.

    In Construction \ref{cnstr:fold1}.(2)-(3), if $\mathbf{e_i^0}$ or $\mathbf{e^0}$ is above or below x-axis, then all new horizontal rectangles are above or below $\mathbf{e_0}$. 
    Hence, if a point $(x,y) \in \widehat{R}(T;a,b)$, then the segment connecting $(x,y)$ to $(x,0)$ also lies in $\widehat{R}(T;a,b)$. 
    Therefore, one can split $\widehat{R}(T;a,b)$ into vertical rectangles by extending all boundary vertical segments. For each vertical rectangle, there is exactly one edge on x-axis. By Lemma \ref{lem:for_fold}, it is not difficult to show by induction that the signed height of $i$-th vertical rectangle is just the sum of signed weights of the first $i$ vertices in $\perm(T;a,b)$. 
    So the vertical rectangle decomposition of $\widehat{R}(T;a,b)$ and $R(\perm(T;a,b))$ are basically the same, differed by an adjustment on the horizontal width of vertical rectangles. 
    We already know the horizontal rectangle decomposition of $\widehat{R}(T;a,b)$, which comes from $(T;a,b)$. Adjusting the width of vertical rectangles does not change the horizontal decomposition. 
    So there is a natural bijection between rectangles in the horizontal decomposition of $R(\perm(T;a,b))$ and $\widehat{R}(T;a,b)$, which preserves the height and relative positions of these rectangles. This implies $G(\perm(T;a,b)) = (T;a,b)$. By the last sentence of Lemma \ref{lem:for_fold}, the two marked vertices in order are $a$ and $b$ again.  

    \medskip
    (2). $\fold(\comb(\perm)) = \perm$ for all $\perm\in\PT(\ppf)$.

    Steps are similar. 
    Let $\ell$ be as stated in Proposition \ref{prop:sign-changing_point}. Then the horizontal rectangles in $\widehat{R}(G(\perm);s_1,s_N)$ from Construction \ref{cnstr:fold1}.(1) are in one-to-one correspondence with horizontal rectangles in $R(\perm)$ bounded by x-axis. The relative positions are kept by this bijection. 
    Next, the order of new horizontal rectangles added in Construction \ref{cnstr:fold1}.(2)-(3) accords with Corollary \ref{cor:rec_length_strict}. By induction steps, the above bijection extends to all horizontal rectangles in $\widehat{R}(G(\perm);s_1,s_N)$, preserving relative positions and heights. So horizontal decomposition of $\widehat{R}(G(\perm);s_1,s_N)$ and $R(\perm)$ are basically the same again. 
    With the bijection between boundary vertical segments of $\widehat{R}(G(\perm);s_1,s_N)$ and $[N]$ deduced from Lemma \ref{lem:for_fold}, $\perm(G(\perm);s_1,s_N) = \perm$. 
\end{proof}

\bigskip
\section{Decomposable case}\label{sec:decomposable}

We have proved Theorem \ref{thm:passport_simple} for non-decomposable passports. For decomposable case, we would like to use strong induction on the maximal partition length (Definition \ref{def:partition}), hence the known results for non-decomposable cases can be used as the initial step. 
But a direct enumeration of $\PT(\ppf)$ is hard to find. Our solution is counting a special class of permutations which allows recursion. 

In this section, we first introduce the set $\PP_+(\ppf)$ of positive permutations, and the set $\PT_+(\ppf)\subset \PP_+(\ppf)$ of positive tree permutations (Definition \ref{def:pos_perm}). The restricted combing map provides the second important bijection, from $\PT_+(\ppf)$ to the set of rooted LWBP-trees (Proposition \ref{prop:pos_and_tree}). Positive permutations admits a well-behaved decomposition, hence inducing a recursive enumeration formula (Proposition \ref{prop:horizontal_division}), in which both terms $\PP_+(*)$ and $\PT_+(*)$ appear. 

Next, an independent enumeration for $\PP_+(\ppf)$ is given. An essential technique is turning a decomposable passport into a non-decomposable one by perturbing its weight function. After chasing its effect on positive permutations, we get a recursive formula (\eqref{eq:Pos_refined_passport} in Lemma \ref{lem:Pos_refined_passport_and_formula}). To finish the induction step, properties for Stirling numbers of the second kind are used.

Finally, with this independent enumeration, we can deduce the formula for $\PT_+(\ppf)$ from the recursive formula of $\PP_+(\ppf)$ (Theorem \ref{thm:number_PosT}). This also completes the proof of main theorem.

\medskip
More notations for partitions are needed in this section, which help us to reveal delicate algebraic relations. 

There is a partial order on $\Ptt(\ppf)$, essential for the enumeration in this Section.
\begin{definition}\label{def:finer_partition}
    In $\Ptt(\ppf)$, a partition $\ptt{q}$ is \textbf{finer than} a partition $\ptt{p} = \{ \pp_i \}_{i=1}^{n}$, denoted as $\ptt{q} \preccurlyeq \ptt{p}$, if there exist partitions $\ptt{p}_i \in \Ptt(\pp_i)$ such that $\ptt{q} = \bigsqcup_{i=1}^{n} {\ptt{p}_i}$ as sets of subpassports of $\ppf$. 
    We also say that $\ptt{p}$ is \textbf{coarser than} $\ptt{q}$. 
    Then $(\Ptt(\ppf) , \preccurlyeq)$ is a partially ordered set.  

    If $\ptt{q} = \bigsqcup_{i=1}^{n} {\ptt{p}_i}$ is finer than $\ptt{p}$ as above, then we see 
\begin{equation}
    X(\ptt{q}) = \prod_{i=1}^{n} X(\ptt{p}_i) \ .
\end{equation} 
\end{definition}

Sometimes we will consider ordered partitions. 
\begin{definition}\label{def:ordd_partition}
    An \textbf{ordered} $n$-partition is an ordered sequence of subpassports 
    $ \ptt{p}^o := (\pp_1, \cdots, \pp_n) $, 
    such that the unordered set $\{ \pp_1, \cdots, \pp_n\}$ is an $n$-partition. 
    Denote $\Ptt_n^o(\ppf)$ as the set of all ordered $n$-partitions of $\ppf$. 
\end{definition}

\bigskip
\subsection{Positive permutation}\label{ssec:pos_perm} 

The special class of permutations we are studying is characterized by the positivity of cumulative sum.

\begin{definition}\label{def:pos_perm}
    For a given full passport $\ppf$,
    \[ \PP_+(\pp_F) := \defset{\perm \in \PP(\ppf)}{H_i>0,\ \forall i\in[N-1]} \]
    is called the set of \textbf{positive permutations}. And define the set of \textbf{positive tree permutations} as
    \[ \PT_+(\ppf) :=  \PP_+(\pp_F) \cap \PT(\ppf) \ .  \]
\end{definition}

\begin{proposition}\label{prop:pos_and_tree}
    There is a bijection from the set of positive tree permutations to the set of rooted trees (Definition \ref{def:twice_marked_tree}), induced by the restricted combing map:   
    \[ \comb : \PT_+(\ppf) \longrightarrow \TR(\ppf) \ . \]
    Therefore, 
    \[ \absbig{\PT_+(\ppf)} = (\absbig{\ppf}-1) \absbig{\Tree(\ppf)}  \ . \]
\end{proposition}
\begin{proof}
    By Proposition \ref{prop:edge_criterion}, for any $\perm=(s_i)_{i=1}^N\in \PT_+(\ppf)$, there is a unique edge connecting $s_1, s_N$. Besides, $\wt(s_1)=H_1>0$ and then $\wt(s_N)<0$. 
    Conversely, for any rooted tree $(T;a,b)\in\TR(\ppf)$, the region $\widehat{R}(T;a,b)$ in Construction \ref{cnstr:fold1} is obviously above the x-axis. 
    
    Now a tree in $\Tree(\ppf)$ always has $(\absbig{\ppf}-1)$ edges. And the black and white vertex of an edge is clearly unique. So $\absbig{\TR(\ppf)}$ is precisely $(\absbig{\ppf}-1)\cdot\absbig{\Tree(\ppf)}$.   
\end{proof}

\begin{corollary}\label{cor:PT_pos_nondecomp}
    Together with Corollary \ref{cor:count_nondecomp}, for non-decomposable $\ppf$, 
    \[ \absbig{\PP_+(\ppf)} = \absbig{\PT_+(\ppf)} = (\absbig{\ppf} - 1)!\ . \eqno\square \] 
\end{corollary}

\medskip
Proposition \ref{prop:pos_and_tree} bridges the enumeration of LWBP-trees and positive permutations. 
The following decomposition result for general positive permutations will instruct the computation. 

Let $\perm=(s_i)_{i=1}^N\in \PP_+(\ppf)$. By Proposition \ref{prop:comb_PT_tree} and its proof, if $G(\perm)$ is not a tree, then there exist $k < l \in [N-1]$ such that $H_{k-1}=H_l >0$ and $\min\{H_k,\cdots,H_{l-1}\} > H_{k-1} = H_l$. Therefore, the extended horizontal segment at height $H_{k-1}$, bounded by the $k$-th and $l$-th vertical line, cuts $R(\perm)$ into two part. Denote the part higher than this horizontal segment as $R_2(\perm)$. 
Define
    \begin{itemize}
        \item two index sets $S_2 := \{s_k, \cdots, s_l\},\ S_1 := [N] \setminus S_2$; 
        \item two subpassports $\pp_i := \pp(S_i)$ of $\ppf$ ($i = 1,2$); 
        \item and two permutations $\perm_1 := (s_1,\cdots,s_{k-1},s_{l+1},\cdots,s_N) \in \PP(\pp_1)$, $\perm_2:=(s_k,\cdots,s_l) \in \PP(\pp_2)$. 
    \end{itemize}
It is not hard to see that the closed region $R(\perm_2)$ and the part $R_2(\perm)$ differ by a translation. $R(\perm_1)$ can be obtained from $R(\perm)$ by two steps. First, remove all vertical rectangles between the $k$-th and the $(l+1)$-th vertical lines; then translate the portion to the right of the $(l+1)$-th vertical line leftward by $l-k+1$. In the proof of Proposition \ref{prop:comb_PT_tree}, we show that there is no edge from $\{k,\cdots,l\}$ to the other vertices. 
So the forest $G(\perm)$ is the disjoint union of $G(\perm_1)$ and $G(\perm_2)$. 
Besides, 
\begin{align*}
    &\min\{\wt(s_k), \cdots, \wt(s_k)+\dots\wt(s_{l-1}) \} \\
    &= \min\{H_{k}-H_{k-1}, \cdots, H_{l-1}-H_{k-1} \} 
    > H_{k-1}-H_{k-1} = H_{l}-H_{k-1} = 0 .    
\end{align*}
So $\perm_2 \in \PP_+(\pp_2)$ and there is an edge between vertex $k$ and $l$ in $G(\perm_2) \in \Fore(\pp_2)$. 
See Figure \ref{fig:Pos_decomp} as an example of this procedure. 

One can keep extracting a subsequences from $\perm_1$, until $G(\perm_1)$ becomes a tree. So the structure of $G(\perm)$ or $R(\perm)$ for general $\perm \in \PP_+(\ppf)$ may be intuitively pictured as follows. There is an underlying tree (of a smaller subpassport) right above the x-axis, with many very long edges; some forests from positive permutations (of smaller subpassports) grow on some of these long edges. See Figure \ref{fig:Pos_decomp} again.

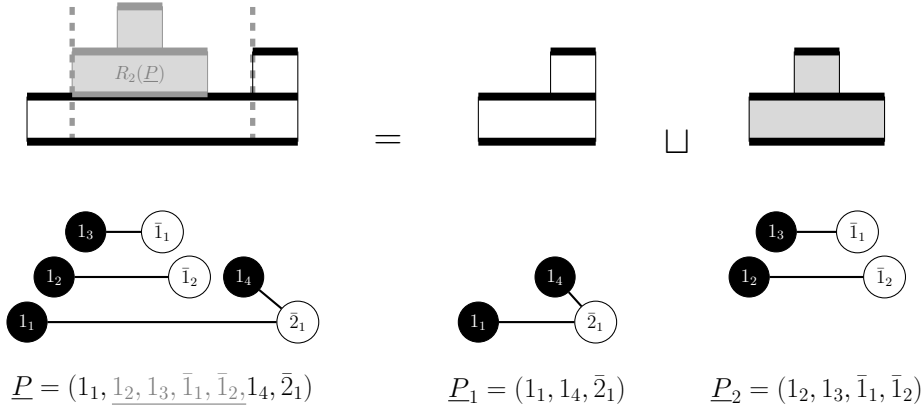
\begin{figure}[t]
\makebox[\textwidth]{
\begin{tikzpicture}[scale=0.6, transform shape]
\fill[lightgray!60] (1 ,1 ) rectangle (4 ,2 ) ;
\fill[lightgray!60] (2 ,2 ) rectangle (3 ,3 ) ;
\draw[line width=2pt, draw=black!40] (1 , 0) edge[dashed] (1 ,3 );
\draw[line width=2pt, draw=black!40] (5 , 0) edge[dashed] (5 ,3 );

\draw (0 ,0 ) edge (0, 1 );
\draw[draw=black!40] (1 ,1 ) edge (1, 2 );
\draw[draw=black!40] (2 ,2 ) edge (2, 3 );
\draw[draw=black!40] (3 ,3 ) edge (3, 2 );
\draw[draw=black!40] (4 ,2 ) edge (4, 1 );
\draw (5 ,1 ) edge (5, 2 );
\draw (6 ,2 ) edge (6, 0 );

\draw[line width=3pt] (0 ,0 ) -- (6, 0 );
\draw[line width=3pt] (0 ,1 ) -- (6, 1 );
\draw[line width=3pt] (5 ,2 ) -- (6, 2 );
\draw[line width=3pt, draw=black!40] (1 ,2 ) -- (4, 2 );
\draw[line width=3pt, draw=black!40] (2 ,3 ) -- (3, 3 );
\draw[line width=2.2pt, draw=black!40] (1 ,1.05) -- (4 ,1.05);

\node[text=gray]() at (2.5,1.5) {$R_2(\perm)$};

\node[blk](11) at (0 ,-4) {$1_1$};
\node[blk](12) at (0.6,-3) {$1_2$};
\node[blk](13) at (1.3,-2) {$1_3$};
\node[wht](14) at (3 ,-2) {$\bar{1}_1$};
\node[wht](15) at (3.6,-3) {$\bar{1}_2$};
\node[blk](16) at (4.8,-3) {$1_4$};
\node[wht](20) at (6 ,-4) {$\bar{2}_1$};
\draw[thick] (11) to (20);
\draw[thick] (12) to (15);
\draw[thick] (13) to (14);
\draw[thick] (16) to (20);

\node at (3 ,-5.5) {\Large $\perm=(1_1, \textcolor{gray}{\underline{1_2, 1_3, \bar{1}_1, \bar{1}_2,}} 1_4, \bar{2}_1)$};

\node at (8,0) {\huge $=$};

\begin{scope}[shift={(10,0)}]
\draw (0 ,0 ) edge (0, 1 );
\draw (1.6,1 ) edge (1.6,2 );
\draw (2.6,2 ) edge (2.6,0 );
\draw[line width=3pt] (0 ,0 ) -- (2.6, 0 );
\draw[line width=3pt] (0 ,1 ) -- (2.6, 1 );
\draw[line width=3pt] (1.6,2 ) -- (2.6,2 );

\node[blk](a11) at (0 ,-4) {$1_1$};
\node[blk](a16) at (1.7,-3) {$1_4$};
\node[wht](a20) at (2.6,-4) {$\bar{2}_1$};
\draw[thick] (a11) to (a20);
\draw[thick] (a16) to (a20);

\node at (1.3,-5.5) {\Large $\perm_1=(1_1, 1_4, \bar{2}_1)$};
\end{scope}

\node at (14.4,0) {\huge $\sqcup$};

\begin{scope}[shift={(16,0)}]
\fill[lightgray!60] (0 ,0 ) rectangle (3 ,1 ) ;
\fill[lightgray!60] (1 ,1 ) rectangle (2 ,2 ) ;
\draw (0 ,0 ) edge (0, 1 );
\draw (1 ,1 ) edge (1, 2 );
\draw (2 ,2 ) edge (2, 1 );
\draw (3 ,1 ) edge (3, 0 );
\draw[line width=3pt] (0 ,0 ) -- (3, 0 );
\draw[line width=3pt] (0 ,1 ) -- (3, 1 );
\draw[line width=3pt] (1 ,2 ) -- (2 ,2 );

\node[blk](b12) at (0 ,-3) {$1_2$};
\node[blk](b13) at (0.6,-2) {$1_3$};
\node[wht](b14) at (2.4,-2) {$\bar{1}_1$};
\node[wht](b15) at (3 ,-3) {$\bar{1}_2$};
\draw[thick] (b12) to (b15);
\draw[thick] (b13) to (b14);

\node at (1.5,-5.5) {\Large $\perm_2=(1_2, 1_3, \bar{1}_1, \bar{1}_2)$};
\end{scope}
\end{tikzpicture}
}
\caption{Decomposition of a positive permutation.}
\label{fig:Pos_decomp}
\end{figure}

With this in mind, we have the following recursion formula for enumeration. 
For an ordered $2$-partition $(\pp_0, \pp_+) \in \Ptt_2^o$, an $n$-partition of the subpassport $\pp_+$ will be denoted by $\ptt{p}_+=\{\pp_i\}_{i=1}^{n}$ here. 
\begin{proposition}\label{prop:horizontal_division}
    For decomposable full passport $\ppf$, 
    \begin{align*}   
        \absbig{\PP_+(\ppf)} &= \absbig{\PT_+(\ppf)} \\ 
        & + \sum_{(\pp_0,\pp_+)\in \Ptt_2^o} {\absbig{\PT_+(\pp_0)} \sum_{\ptt{p}_+ \in \Ptt(\pp_+)} \frac{(\absbig{\pp_0}+\absbig{\ptt{p}_+}-2)!}{(\absbig{\pp_0}-2)!} \prod_{i = 1}^{\absbig{\ptt{p}_+}} {\absbig{\PP_+(\pp_i)}}} \ .
    \end{align*}
\end{proposition}

\begin{proof}
Let $\perm\in\PP_+(\ppf)\setminus\PT_+(\ppf)$. 
If $k < l \in [N-1]$ with $\min\{H_k,\cdots,H_{l-1}\} > H_{k-1} = H_l >0$, then $(k, l)$ is called a \textbf{horizontal pair}. 
A horizontal pair is \textbf{maximal} if there is no other horizontal pair $(k', l')$ with $k'<k$ and $l'>l$. 

For each maximal horizontal pair $(k_i, l_i)$, define $S_i := \{s_j\}_{j=k_i}^{l_i}$, a subpassport $\pp_i := \pp(S_i)$ and a sequence $\perm_i := (s_j)_{j=k_i}^{l_i} \in \PP_{+}(\pp_i)$. 
The index sets $S_i$ of different maximal horizontal pair are disjoint. 
Then define $\pp_0:= \pp(S\setminus\sqcup S_i)$, with the union running through all maximal horizontal pairs. 
Let $\perm_0 \in \PP(\pp_0)$ be the remained sequence obtained by removing all subsequences $\perm_i$ from $\perm$. By previous argument, we have $\perm_0 \in \PT_+(\pp_0)$. $s_1,\ s_N$ are the first and last element of $\perm_0$ again. And $G(\perm)$ is the disjoint union of a tree $G(\perm_0)$ and several forests $G(\perm_i)$. See Figure \ref{fig:Pos_decomp}. 

To uniquely recover such a permutation $\perm\in\PP_+(\ppf)\setminus\PT_+(\ppf)$, we first choose a subpassport $\pp_0 := \pp(S_0)$ induced by a non-empty proper subset $S_0 \subsetneq S$, and a permutation $\perm_0\in \PT_+(\pp_0)$. 
We also pick an $n$-partition $\ptt{p}_+=\{\pp_1,\cdots,\pp_n\}$ of the  subpassport $\pp_+ := \pp(S \setminus S_0)$, together with $n$ permutations $\perm_i\in\PP_+(\pp_i)$. 
Finally, we insert the later $n$ permutations into $\perm_0$ one by one. 
$\perm_1$ can be inserted between $s_k$ and $s_{k+1}$ of $\perm_0$ for any $1\leq k < k+1 \leq |\pp_0|$. So there are $|\pp_0| - 1$ insertion positions for $\perm_1$. 
The next permutation $\perm_2$ must not be inserted between any labels of $\perm_1$, so there are $|\pp_0|$ positions. 
Similarly, the number of insertion positions increases by one for each new permutations. 
So for each fixed $\pp_0$ and $\ptt{p}_+$, the number of choices in total is just the permutation number $(|\pp_0|+|\ptt{p}_+|-2)!/(|\pp_0|-2)!$. 
\end{proof}

\begin{remark} \label{rmk:horizontal_division_general}
There is a decomposition result for general $\perm$. We need negative permutations analogous to Definition \ref{def:pos_perm}.
One first splits $\perm$ into several subsequences at the zeros of $H_i$'s. For each subsequence, after extracting all maximal horizontal pairs bounding a positive or negative permutation, the remaining is a tree permutation. 
\hfill $\square$
\end{remark}

\bigskip
\subsection{Perturbed passport}\label{ssec:perturb} 

Now we count positive permutations. The key idea is turning a decomposable passport to a non-decomposable one by perturbing its weight function, which is very similar to the ``$\varepsilon$-edge'' technique in \cite{YYK15}. 
Here is the main result of this subsection. 

\begin{theorem} \label{thm:number_Pos}
    \begin{equation} \label{eq:number_Pos}
        \absbig{\PP_+(\ppf)} = \sum_{\ptt{p} \in \Ptt(\ppf)} {(-1)^{\absbig{\ptt{p}} - 1} X(\ptt{p})} .
    \end{equation}
\end{theorem}
Recall that $X(\ptt{p}) = \prod_{i=1}^{n} \left( \absbig{ \pp_i } - 1 \right)!$ in Definition \ref{def:partition}.

\medskip
\begin{definition} \label{def:perturbed_passport}
    Fix a decomposable full passport $\ppf = (S, \one, \wt)$ with $\absbig{\ppf}=N$. A \textbf{perturbed passport} of $\ppf$ will be a non-decomposable one $\widetilde{\pp}_F = (S, \one, \widetilde{\wt})$ with the same index set. Its weight function $\widetilde{\wt}$ is defined by following steps.
    \begin{enumerate}
        \item Let 
        \[ E_0 = \min \defset{\absbig{\textstyle\sum_{s \in A} {\wt(s)}}}{A \subseteq S \text{ s.t. } \textstyle\sum_{s \in A} {\wt(s)} \neq 0}  ,\quad \varepsilon_0 = \frac{E_0}{N - 1}. \] 
        \item Choose $s_- \in S$ and a function $\varepsilon : S \setminus \{s_-\} \to (0, \varepsilon_0)$. 
        Then let 
        \[ \varepsilon(s_-) := - \textstyle\sum_{s \neq s_-}  {\varepsilon(s)} \in (- E_0 , 0) . \] 
        \item Define $\widetilde{\wt}(s) = \wt(s) + \varepsilon(s)$ for all $s \in S$.
    \end{enumerate}
\end{definition}

\begin{lemma} \label{lem:refined_passport_epsilon_function}
    For the function $\varepsilon(s)$ above, 
    \begin{enumerate}
        \item $\displaystyle\absbig{\sum_{s \in A} {\varepsilon(s)}} < E_0$ for any subset $A \subseteq S$;
        \item When $\varnothing \neq A \subsetneq S$, $\displaystyle\sum_{s \in A} \varepsilon(s) \neq 0$. And  
        $\displaystyle\sum_{s \in A} \varepsilon(s) > 0 \Leftrightarrow s_- \notin A$.
    \end{enumerate}
    Then $\widetilde{\pp}_F = (S, \one, \widetilde{\wt})$ is non-decomposable. 
\end{lemma}
\begin{proof}
    Since $\displaystyle\sum_{s \in S} {\varepsilon(s)} = 0$, we know that $\displaystyle\sum_{s \in S} {\widetilde{\wt}(s)} = \sum_{s \in S} \wt(s) + \sum_{s \in S} \varepsilon(s) = 0$. Thus $\widetilde{\pp}_F$ is a passport indeed. 
    The two properties of $\varepsilon$ are easy to verify. 

    Now assume $A$ is a non-empty proper subset of $S$. 
    If $\sum_{s \in A} {\wt(s)} \neq 0$, then by Lemma \ref{lem:refined_passport_epsilon_function}.(1) and our choice of $E_0$, 
    \begin{equation*}
        \absbig{\sum_{s \in A} {\widetilde{\wt}(s)}} \geq \absbig{\sum_{s \in A} {\wt(s)}} - \absbig{\sum_{s \in A} {\varepsilon(s)}} > \absbig{\sum_{s \in A} {\wt(s)}} - E_0 \geq 0 \ .
    \end{equation*}
    If $\sum_{s \in A} {\wt(s)} = 0$, by Lemma \ref{lem:refined_passport_epsilon_function}.(2),  
    $\sum_{s \in A} {\widetilde{\wt}(s)} = \sum_{s \in A} {\varepsilon(s)} \neq 0$. Therefore, $\widetilde{\pp}_F$ is non-decomposable.
\end{proof}

\medskip
Now we fix a perturbed passport $\widetilde{\pp}_F$ of the decomposable passport $\ppf$. Since they share the same index set $S$, $\PP(\widetilde{\pp}_F)$ is naturally identified with $\PP(\ppf)$.

\begin{lemma} \label{lem:Pos_refined_passport_and_formula}
    Let 
    \[ \PP_{\geq0}(\pp_F) := \defset{\perm \in \PP(\ppf)}{H_i\geq0,\ \forall i\in[N-1]} . \]    
    \begin{enumerate}
        \item $\PP_+(\widetilde{\pp}_F) \subseteq \PP_{\geq 0}(\ppf)$. 
        \item 
        \begin{equation}\label{eq:pos_pertubed}
            \PP_+(\widetilde{\pp}_F) \ \cong \ 
            \bigsqcup_{n=1}^{m(\ppf)} \bigsqcup_{\substack{\ptt{p}^o \in \Ptt^o_n(\ppf) \\ s_- \in S_n}} {\prod_{i = 1}^{n} \PP_+(\pp_i)} \ .
        \end{equation}
        Here $\ptt{p}^o = (\pp_i)_{i = 1}^n \in \Ptt^o_n(\ppf)$ with $\pp_i = \pp(S_i) \subset \ppf$ . 
        \item 
        \begin{equation} \label{eq:Pos_refined_passport}
            (\absbig{\ppf} - 1)! \  = \sum_{\ptt{p} \in \Ptt(\ppf)} {(\absbig{\ptt{p}} - 1)! \prod_{i = 1}^{\absbig{\ptt{p}}} {\absbig{\PP_+(\pp_i)}}} .
        \end{equation}
    \end{enumerate}
\end{lemma}

\begin{proof}
    (1)
    For a permutation $\perm \in \PP(\ppf) \cong \PP(\widetilde{\pp}_F)$, define the cumulative sums as $H_i := \sum_{j=1}^{i} \wt(s_j)$ and $\widetilde{H}_i := \sum_{j=1}^{i} \widetilde{\wt}(s_j)$. 
    Assume $\perm \in \PP_+(\widetilde{\pp}_F)$, then with Lemma \ref{lem:refined_passport_epsilon_function}.(1),
    $H_i=\widetilde{H}_i - \sum_{j=1}^{i}\varepsilon(s_j) > 0 - E_0$ for all $i\in[N-1]$. By the choice of $E_0$ in Definition \ref{def:perturbed_passport}, 
    $H_i > - E_0$ implies $H_i \geq 0$.

    \medskip
    (2)
    By studying the connectedness of the region $R(\perm)$, we easily see
    \[ \PP_{\geq0}(\ppf) \ \cong\ \bigsqcup_{n=1}^{m(\ppf)} \bigsqcup_{\ptt{p}^o \in \Ptt^o_n(\ppf)} {\prod_{i = 1}^{n} \PP_+(\pp_i)} . \]

    Now assume $\perm \in \PP_{\geq0}(\ppf)$ is identified with $(\perm_1,\cdots,\perm_n)$, where $\perm_i\in \PP_+(\pp_i)$, $\pp_i = \pp(S_i)$ for $1\leq i \leq n$, and $\ptt{p}^o:=(\pp_1,\cdots,\pp_n) \in \Ptt_n^o(\ppf)$. 
    Let $H_i, \widetilde{H}_i$ as above. Then with similar arguments, we see that $H_i > 0$ implies $\widetilde{H}_i > 0$. 
    So we only need to consider the subscript set
    \[ I:=\{ |S_1|, |S_1\sqcup S_2|, \cdots, |S_1\sqcup\dots S_{n-1}| \} = \defset{i\in[N-1]}{H_i=0} .\]
    It remains to prove that
    \begin{equation*}
        \widetilde{H}_{i} > 0 \textrm{ for all } i\in I \quad \Longleftrightarrow \quad s_- \in S_n \quad .
    \end{equation*}
    For each $1\leq k \leq n-1$, let $i_k:=|S_1 \sqcup\dots S_k|\in I$. Then $\widetilde{H}_{i_k} = \sum_{s\in S_1\sqcup\dots S_k} \varepsilon(s)$. 
    By Lemma \ref{lem:refined_passport_epsilon_function}.(2), it is positive if and only if $s_- \notin S_1 \sqcup\dots S_k$. 
    This happens for all $1\leq k \leq n-1$ if and only if $s_- \in S_n$.

    \medskip
    (3) 
    This is a direct enumeration of \eqref{eq:pos_pertubed}. 
    Since $\widetilde{\pp}_F$ is non-decomposable, 
    $|{\PP_+(\widetilde{\pp}_F)}| = (|\ppf|-1)!$ by Corollary \ref{cor:PT_pos_nondecomp}.
    
    On the other hand, given an $n$-partition $\ptt{p}\in\Ptt_n(\ppf)$, there are exactly $(\absbig{\ptt{p}} - 1)!$ choices of ordered version satisfying $s_- \in S_n$. 
\end{proof}

For enumeration of $\PP_+(\ppf)$, we introduce the formal function on passport
\begin{equation}\label{eq:NPP+}
    NP_+(\ppf) := \sum_{\ptt{p} \in \Ptt(\ppf)} {(-1)^{\absbig{\ptt{p}} - 1} X(\ptt{p})}
\end{equation}
and study its properties. 
Then we prove $\absbig{\PP_+(\ppf)} = NP_+(\ppf)$. Hence $\absbig{\PP_+(\ppf)}$ will possess the same properties as the function $NP_+(\ppf)$. 
Some combinatorial quantities will be used next.

Let $x$ be a real variable. The \textbf{falling and rising factorials} are defined as 
\newcommand{\ffact}[2]{\big( #1 \big)_{#2}}
\newcommand{\rfact}[2]{#1 ^{( #2 )} }
\begin{equation}
    \ffact{x}{n}:=x(x-1)\cdots(x-n+1), \quad \rfact{x}{n}:=x(x+1)\cdots(x+n-1) .
\end{equation}
\textbf{Stirling number of the second kind} $S(n,m)$ is the number of $m$-partitions of an $n$-element set. When $n\in\Zpos$, a relation between them is given by
\begin{equation}\label{eq:stirling2_factorial}
    x^n = \sum_{k=1}^{n} S(n,k) \ffact{x}{k} \quad \Longleftrightarrow \quad (-x)^n=\sum_{k=1}^{n} S(n,k) (-1)^{k} \rfact{x}{k} .
\end{equation}

\begin{lemma} \label{lem:rf_Pos_multi_eq_power_X}
    Let $x$ be a real variable and $NP_+(\ppf)$ defined as \eqref{eq:NPP+}. Then 
    \begin{equation} \label{eq:rf_Pos_multi_eq_power_X}
        \sum_{\ptt{p} \in \Ptt(\ppf)} \rfact{x}{|\ptt{p}|} \prod_{i = 1}^{\absbig{\ptt{p}}} {NP_+(\pp_i)} = \sum_{\ptt{p} \in \Ptt(\ppf)} {x^{\absbig{\ptt{p}}} X(\ptt{p})}
    \end{equation}
\end{lemma}

\begin{proof}

    Given $\ptt{p}=\{\pp_i\}_{i=1}^n \in \Ptt(\ppf)$ and 
    $\ptt{p_i}\in\Ptt(\pp_i)$, $1 \leq i \leq n$, 
    then $\ptt{q}:=\bigsqcup_{i=1}^n \ptt{p}_i$ is obviously a partition of $\ppf$ finer than $\ptt{p}$. So
    \begin{align*}
        \textrm{LHS of \eqref{eq:rf_Pos_multi_eq_power_X}} 
        =& \sum_{\ptt{p} \in \Ptt(\ppf)} {\rfact{x}{|\ptt{p}|} \prod_{i = 1}^{\absbig{\ptt{p}}} { \sum_{\ptt{p}_i \in \Ptt(\pp_i)} {(-1)^{|\ptt{p}_i| - 1} X(\ptt{p}_i)} }} \\
        =& \sum_{\ptt{p} \in \Ptt(\ppf)} {\rfact{x}{|\ptt{p}|} \sum_{\substack{\ptt{q} \in \Ptt(\ppf) \\ \ptt{q} \preccurlyeq \ptt{p}}} {(-1)^{\absbig{\ptt{q}} - \absbig{\ptt{p}}} } X(\ptt{q})} \\ 
        =& \sum_{\ptt{q} \in \Ptt(\ppf)} (-1)^{|\ptt{q}|} X(\ptt{q}) \sum_{\substack{\ptt{p} \in \Ptt(\ppf) \\ \ptt{q} \preccurlyeq \ptt{p}}} 
        (-1)^{|\ptt{p}|} \rfact{x}{|\ptt{p}|} .
    \end{align*}
    For a given partition $\ptt{q}\in\Ptt(\ppf)$ and $1 \leq k \leq |\ptt{q}|$, a $k$-partition coarser than $\ptt{q}$ corresponds to a $k$-partition of $\ptt{q}$ as a set. Hence
    \[ \absbig{ \defset {\ptt{p}\in\Ptt_k(\ppf)} {\ptt{q}\preccurlyeq\ptt{p} } } = S(|\ptt{q}|,k)  \]
    and 
    \begin{align*}
        \sum_{\substack{\ptt{p} \in \Ptt(\ppf) \\ \ptt{q} \preccurlyeq \ptt{p}}} 
        (-1)^{|\ptt{p}|} \rfact{x}{|\ptt{p}|} 
        &= \sum_{k=1}^{|\ptt{q}|} \sum_{\substack{\ptt{p} \in \Ptt_k(\ppf) \\ \ptt{q} \preccurlyeq \ptt{p}}} (-1)^{k} \rfact{x}{k} \\
        &= \sum_{k=1}^{|\ptt{q}|} S(|\ptt{q}|,k) (-1)^{k} \rfact{x}{k}
        \xlongequal{\eqref{eq:stirling2_factorial}} (-x)^{|\ptt{q}|} .
    \end{align*}
    Then \eqref{eq:rf_Pos_multi_eq_power_X} is obtained. 
\end{proof}

\begin{corollary} \label{cor:rf_Pos_multi_eq_power_X_cancel}
    \begin{equation} \label{eq:rf_Pos_multi_eq_power_X_particular}
        \sum_{\ptt{p} \in \Ptt(\ppf)} {(|\ptt{p}| - 1)! \prod_{i = 1}^{|\ptt{p}|} {NP_+(\pp_i)}} = (\absbig{\ppf} - 1)! \ .
    \end{equation}
\end{corollary}

\begin{proof}
    Since $|\ptt{p}| \geq 1$ for any $\ptt{p} \in \Ptt(\ppf)$, $\rfact{x}{|\ptt{p}|} = x \cdot \rfact{(x + 1)}{|\ptt{p}| - 1}$. 
    By canceling the common factor $x$ from both sides of \eqref{eq:rf_Pos_multi_eq_power_X}, we obtain 
    \[ \sum_{\ptt{p} \in \Ptt(\ppf)} \rfact{(x + 1)}{|\ptt{p}| - 1} \prod_{i = 1}^{|\ptt{p}|} {NP_+(\pp_i)} = \sum_{\ptt{p} \in \Ptt(\ppf)} {x^{|\ptt{p}| - 1} X(\ptt{p})} \ . \]
    Now let $x=0$, then $\rfact{1}{|\ptt{p}|-1} = (|\ptt{p}|-1)!$. And $x^{|\ptt{p}| - 1} X(\ptt{p}) \neq 0 \Leftrightarrow \absbig{\ptt{p}} = 1 \Leftrightarrow \ptt{p} = \ptt{e} $. By Definition \ref{def:partition}.2, $X(\ptt{e}) = (|\ppf|-1)!$ 
\end{proof}

\medskip
Now we are ready to enumerate $\PP_+(\ppf)$.
\begin{proof}[Proof of Theorem \ref{thm:number_Pos}]
    We shall prove $\absbig{\PP_+(\ppf)} = NP_+(\ppf)$ by strong induction on the maximal partition length $m(\ppf)$. 
    When $m(\ppf) = 1$, $\ppf$ is non-decomposable. By Corollary \ref{cor:PT_pos_nondecomp}, 
    \[ \absbig{\PP_+(\ppf)} = 
    (\absbig{\ppf}-1)! = X(\ptt{e}) = NP_+(\ppf) . \]
    Now assume $m(\ppf)>1$.
    For each partition $\ptt{p} = \{\pp_i\}_{i=1}^n$ whose length $\absbig{\ptt{p}} > 1$, it is easy to see $m(\pp_i) < m(\ppf), \forall\  1 \leq i \leq n$. 
    Now apply the induction hypothesis on \eqref{eq:Pos_refined_passport}, we see  
    \begin{equation*}
        (\absbig{\ppf} - 1)! = \absbig{\PP_+(\ppf)} + \sum_{\substack{\ptt{p} \in \Ptt(\ppf) \\ |\ptt{p}| > 1}} {(|\ptt{p}| - 1)! \prod_{i = 1}^{|\ptt{p}|} {NP_+(\pp_i)}} .
    \end{equation*}
    Comparing this with \eqref{eq:rf_Pos_multi_eq_power_X_particular}, we immediately solve  $\absbig{\PP_+(\ppf)} = NP_+(\ppf)$.
\end{proof}

\bigskip
\subsection{Enumerating positive tree permutations}\label{ssec:decomp_count} 
Finally we are ready to enumerate $\PT_+(\ppf)$. This completes the proof of Theorem \ref{thm:passport_simple}. 

\begin{theorem} \label{thm:number_PosT}
    \begin{equation} \label{eq:number_PosT}
        \absbig{\PT_+(\ppf)} = \sum_{\ptt{p} \in \Ptt(\ppf)} {(-1)^{\absbig{\ptt{p}} - 1} (\absbig{\ppf} - 1)^{\absbig{\ptt{p}} - 1} X(\ptt{p})} \ .
    \end{equation}
\end{theorem}

The following lemma is used in the proof of Theorem \ref{thm:number_PosT}.

\begin{lemma} \label{lem:subset_sum_power}
    For $m \in \Zpos$, integer $0 \leq k \leq m-1$, real variables $y$ and 
    $\{x_i\}_{i = 1}^m$, we have
    \begin{equation}
        \sum_{A \subseteq [m]} {(-1)^{\absbig{A}} \left(y + \sum_{i \in A} {x_i}\right)^k} = 0
    \end{equation}
\end{lemma}

\begin{proof}
    Consider the generating function
    \begin{equation}
        f(t) := \ee^{y t} \prod_{i = 1}^m \left(1 - \ee^{x_i t}\right) = \sum_{A \subseteq [m]} {(-1)^{\absbig{A}} \cdot \ee^{\left(y + \sum_{i \in A} {x_i}\right) t}} \ .
    \end{equation}
    Using the latter expression, the coefficient of the $k$-th term in the Taylor expansion of $f(t)$ at $t=0$ is 
    \begin{equation}
        [t^k]f = \frac1{k!} \sum_{A \subseteq [m]} {(-1)^{\absbig{A}} \left(y + \sum_{i \in A} {x_i}\right)^k} \ .
    \end{equation}
    But according to the former expression, $f(t)$ has a zero of multiplicity $m$ at $t=0$.
    Thus $[t^k]f = 0$ for all integer $0 \leq k \leq m-1$. 
\end{proof}

\medskip
\begin{proof}[Proof of Theorem \ref{thm:number_PosT}]
    We apply strong induction on $m(\ppf)$ again. When $m(\ppf) = 1$, $\ppf$ is non-decomposable and the formula is just Corollary \ref{cor:PT_pos_nondecomp}. 
    So assume $m(\ppf)>1$ next. 

    We plan to compute $|\PT_+(\ppf)|$ by Proposition \ref{prop:horizontal_division}. 
    Note that the permutation number there is just rising factorial $(|\pp_0|-1)^{(|\ptt{p}_+|)}$. 
    For a given ordered 2-partition $(\pp_0,\pp_+) \in \Ptt_2^o(\ppf)$, applying Theorem \ref{thm:number_Pos} and Lemma \ref{lem:rf_Pos_multi_eq_power_X} on $\pp_+$ with $x=|\pp_0|-1$, we see
    \[ \sum_{\ptt{p}_+ \in \Ptt(\pp_+)} \rfact{(|\pp_0|-1)}{|\ptt{p}_+|} \prod_{i=1}^{|\ptt{p}_+|} \absbig{\PP_+(\pp_i)}  =   \sum_{\ptt{p}_+ \in \Ptt(\pp_+)} (|\pp_0|-1)^{|\ptt{p}_+|} X(\ptt{p}_+) . \]
    Together with the induction hypothesis, we have
    \begin{align*}
    & \sum_{(\pp_0,\pp_+)\in \Ptt_2^o} {\absbig{\PT_+(\pp_0)} \sum_{\ptt{p}_+ \in \Ptt(\pp_+)} \frac{(\absbig{\pp_0}+\absbig{\ptt{p}_+}-2)!}{(\absbig{\pp_0}-2)!} \prod_{i = 1}^{\absbig{\ptt{p}_+}} {\absbig{\PP_+(\pp_i)}}} \\
    =& \sum_{(\pp_0,\pp_+)\in \Ptt_2^o} \left(
    \sum_{\ptt{p}_0\in\Ptt(\pp_0)} (-1)^{|\ptt{p}_0|-1} (|\pp_0|-1)^{|\ptt{p}_0|-1} X(\ptt{p}_0)
    \right) \left(
    \sum_{\ptt{p}_+ \in \Ptt(\pp_+)} (|\pp_0|-1)^{|\ptt{p}_+|} X(\ptt{p}_+)
    \right) \\
    =& \sum_{(\pp_0,\pp_+)\in \Ptt_2^o} 
    \sum_{\substack{\ptt{p} = \ptt{p}_0\sqcup \ptt{p}_+ \\ \ptt{p}_0\in\Ptt(\pp_0) \\ \ptt{p}_+\in\Ptt(\pp_+)}} 
    (-1)^{|\ptt{p}_0|-1} (|\pp_0|-1)^{|\ptt{p}|-1} X(\ptt{p}) .
    \end{align*}
    Notice that only partitions of length greater than 1 contribute to the sum. 

    Now fix a partition $\ptt{p} = \{\pp_i\}_{i=1}^{m} \in \Ptt(\ppf)$ with $m>1$. Each ordered 2-partition $(\pp_0,\pp_+)$ coarser than $\ptt{p}$ corresponds to an ordered 2-partition of $\ptt{p}$ as a set, or a proper non-empty subset of $\ptt{p}$. Hence, the coefficient of $X(\ptt{p})$ in the above summation is
    \[ Co(\ptt{p}) := \sum_{\varnothing \ne A\subsetneq[m]} (-1)^{|A|-1} \bigg( \sum_{i \in A}|\pp_i|-1 \bigg)^{|\ptt{p}|-1} . \]
    
    Let $m=|\ptt{p}|,\ k=m-1,\ y=-1,\ x_i=|\pp_i|$ in Lemma \ref{lem:subset_sum_power}. Then we have
    \[ \sum_{A \subseteq [m]} (-1)^{|A|} \bigg( \sum_{i\in A} |\pp_i| - 1 \bigg) ^{|\ptt{p}|-1} = 0. \]
    Therefore, 
    \[ Co(\ptt{p}) = \sum_{A=\varnothing,[m]} (-1)^{|A|} \bigg( \sum_{i\in A} |\pp_i| - 1 \bigg) ^{|\ptt{p}|-1} 
    = (-1)^{|\ptt{p}|-1} + (-1)^{|\ptt{p}|}(|\ppf|-1)^{|\ptt{p}|-1} . \]

    Finally, substituting all the above results into Proposition \ref{prop:horizontal_division}, together with Theorem \ref{thm:number_Pos}, we see
    
    \begin{align*}
        |\PT_+(\ppf)| = & |\PP_+(\ppf)| - \sum_{\ptt{e}\neq\ptt{p}\in\Ptt(\ppf)} Co(\ptt{p}) X(\ptt{p}) \\
        = & \sum_{\ptt{p}\in\Ptt(\ppf)}(-1)^{|\ptt{p}|-1}X(\ptt{p}) - 
        \sum_{\ptt{e}\neq\ptt{p}\in\Ptt(\ppf)} \bigg((-1)^{|\ptt{p}|-1} + (-1)^{|\ptt{p}|}(|\ppf|-1)^{|\ptt{p}|-1} \bigg) X(\ptt{p}) \\
        = & \quad X(\ptt{e}) + \sum_{\ptt{e}\neq\ptt{p}\in\Ptt(\ppf)} \bigg( (-1)^{|\ptt{p}|-1}(|\ppf|-1)^{|\ptt{p}|-1} \bigg) X(\ptt{p}) \\
        = & \sum_{\ptt{p}\in\Ptt(\ppf)} (-1)^{|\ptt{p}|-1}(|\ppf|-1)^{|\ptt{p}|-1} X(\ptt{p}) .
    \end{align*}
\end{proof}

\begin{remark}\label{rmk:proof_without_in-ex}
Section \ref{sec:permutation} answers the last part of the problem ``\textsc{Enumeration according to a passport}'' in \cite[Chapter 12]{APZ2020}. However, for its first part, we suspect it is hard to prove Theorem \ref{thm:passport_simple} for decomposable cases without resorting to inclusion-exclusion. Here are some clues and comments. 

The alternating signs in \eqref{eq:enumeration_passport_decomposable} originates from those in formula \eqref{eq:number_Pos} of $\absbig{\PP_+(\ppf)}$. We remark that \eqref{eq:number_Pos} can also be proved by analyzing the cyclic shifts of a sequence, which is analogous to cycle lemma. 
From this perspective, the result is essentially an instance of the inclusion-exclusion principle. 

In contrast, if we consider a suitably enlarged set, the result will not contain alternating sign. Here are two examples. 
\begin{enumerate}
    \item We can prove the following enumeration formula for non-negative permutations (defined in Lemma \ref{lem:Pos_refined_passport_and_formula}): 
    $$\textstyle \absbig{\PP_{\geq0}(\pp_F)} = \sum_{\ptt{p} \in \Ptt(\ppf)} {X(\ptt{p})} \ .$$
    \item One can also enumerate LWBP-trees that allow zero-weighted edges: 
    $$\textstyle \sum_{\ptt{p} \in \Ptt(\ppf)} (\absbig{\ppf} - 1)^{\absbig{\ptt{p}} - 2} {X(\ptt{p})} \ .$$
    This formula is implicit in Tutte's result \cite{Tutte1970}.
\end{enumerate} 
\hfill $\square$
\end{remark}

\bigskip 
\bibliographystyle{plain}
\bibliography{ref_Tree}

\bigskip

\begingroup
\footnotesize 

Sicheng Lu
 
\textsc{School of Mathematical Sciences, Soochow University, Suzhou, Jiangsu, People's Republic of China.} 

\textit{Email address}: \texttt{\textcolor[rgb]{0.00,0.00,0.84}{sclu@suda.edu.cn}}

\bigskip
Yi Song (Corresponding author) 

\textsc{School of Mathematical Sciences, University of Science and Technology of China, Hefei, Anhui, People's Republic of China.} 

\textit{Email address}: \texttt{\textcolor[rgb]{0.00,0.00,0.84}{sif4delta0@mail.ustc.edu.cn}}

\endgroup

\bigskip

\begin{appendix}

\section{Computation examples}\label{sec:App}
Here we present some concrete examples of our bijection and some numerical results of Kochetkov's formula.

\subsection{A list of trees}\label{ssec:A1}

Considers a decomposable full passport $\ppf = (3\ 1_1\ 1_2\ \ol{4}\ \ol{1})$, following Notation \ref{not:passport}. 
We use the bijection $\PT_+(\ppf) \cong \TR(\ppf)$ in Proposition \ref{prop:pos_and_tree} to list all rooted trees, and therefore the set $\Tree(\ppf)$. 

There are 8 permutations in $\PT_+(\ppf)$, 4 of whom are listed in Figure \ref{fig:the_four_rooted_trees}.  
The corresponding rooted trees are drawn below the permutations, in a manner similar to Figure \ref{fig:folding}, indicating the reverse process. 
After forgetting the rooted edge, they all reduce to the same tree $T_1$ in $\Tree(\ppf)$. 
In particular, $G(\perm_i)$ is rooted at edge $e_i$ of $T_1$, for all $i=1,2,3,4$.

The remaining 4 permutations or rooted trees are obtained by swapping $1_1$ and $1_2$. The induced tree $T_2$ is the mirror copy of $T_1$. Then $\Tree(\ppf)=\{T_1, T_2\}$.

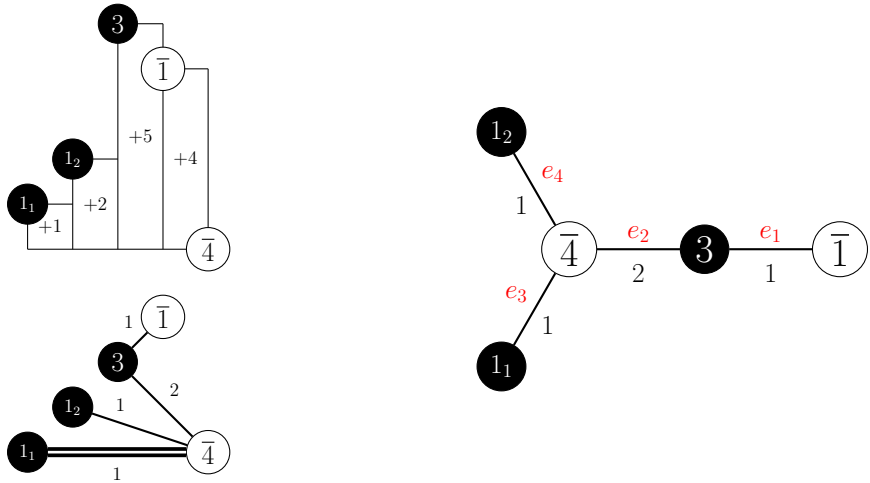
\begin{figure}[ht]
    \centering
\makebox[\textwidth]{
\begin{tikzpicture}[scale=0.6, transform shape]
    \node at (1.2 ,6.5) {\Large $\perm_1 = (3\ ,\ 1_1\ ,\ 1_2\ ,\ \ol{4}\ ,\ \ol{1})$};
    
    \draw (0,0) -- (4,0); 

    \draw (0,0) -- (0,3);
    \draw (1,0) -- (1,4); 
    \draw (2,0) -- (2,5);
    \draw (3,0) -- (3,5);
    \draw (4,0) -- (4,1);

    \draw (0,3) -- (1,3);
    \draw (1,4) -- (2,4);
    \draw (2,5) -- (3,5);
    \draw (3,1) -- (4,1);
    
    \node[blk] at (0,3) {\Large $3$}; 
    \node[blk] at (1,4) {       $1_1$};
    \node[blk] at (2,5) {       $1_2$};
    \node[wht] at (3,1) {\Large $\ol{4}$};
    \node[wht] at (4,0) {\Large $\ol{1}$}; 
    \node at (0.5,1.5) {+3};
    \node at (1.5,2) {+4};
    \node at (2.5,2.5) {+5};
    \node at (3.5,0.5) {+1};

    \begin{scope}[shift={(0,-4.5)}]
    \node[blk] (11) at (0,0) {\Large $3$}; 
    \node[blk] (12) at (1,2) {       $1_1$};
    \node[blk] (13) at (2,3) {       $1_2$};
    \node[wht] (14) at (3,1) {\Large $\ol{4}$};
    \node[wht] (15) at (4,0) {\Large $\ol{1}$}; 

    \draw[line width=1.4pt, double=white, double distance=1pt, black] (11) -- node[pos=0.5, below=5pt] {1} (15);
    \draw[thick] (11) -- node[pos=0.5, above=2pt] {2} (14);
    \draw[thick] (12) -- node[pos=0.5, above=5pt] {1} (14);
    \draw[thick] (13) -- node[pos=0.9, above=10pt] {1} (14);
    \end{scope}
    
    \begin{scope}[shift={(8,0)}]
    \node at (1.2 ,6.5) {\Large $\perm_2 = (3\ ,\ \ol{1}\ ,\ 1_1\ ,\ 1_2\ ,\ \ol{4})$};
    
    \draw (0,0) -- (4,0); 

    \draw (0,0) -- (0,3);
    \draw (1,0) -- (1,3); 
    \draw (2,0) -- (2,3);
    \draw (3,0) -- (3,4);
    \draw (4,0) -- (4,4);

    \draw (0,3) -- (1,3);
    \draw (1,2) -- (2,2);
    \draw (2,3) -- (3,3);
    \draw (3,4) -- (4,4);
    
    \node[blk] at (0,3) {\Large $3$}; 
    \node[wht] at (1,2) {\Large $\ol{1}$}; 
    \node[blk] at (2,3) {       $1_1$};
    \node[blk] at (3,4) {       $1_2$};
    \node[wht] at (4,0) {\Large $\ol{4}$};
    \node at (0.5,1.5) {+3};
    \node at (1.5,1) {+2};
    \node at (2.5,1.5) {+3};
    \node at (3.5,2) {+4};

    \begin{scope}[shift={(0,-4.5)}]
    \node[blk] (21) at (0,0) {\Large $3$}; 
    \node[wht] (22) at (1,2) {\Large $\ol{1}$}; 
    \node[blk] (23) at (2,1) {       $1_1$};
    \node[blk] (24) at (3,2) {       $1_2$};
    \node[wht] (25) at (4,0) {\Large $\ol{4}$};

    \draw[line width=1.4pt, double=white, double distance=1pt, black] (21) -- node[pos=0.5, below=5pt] {2} (25);
    \draw[thick] (21) -- node[pos=0.1, above=10pt] {1} (22);
    \draw[thick] (23) -- node[pos=0.5, above=5pt] {1} (25);
    \draw[thick] (24) -- node[pos=0.9, above=10pt] {1} (25);
    \end{scope}
    \end{scope}

    \begin{scope}[shift={(16, 0)}]
    \node at (1.2 ,6.5) {\Large $\perm_3 = (1_2\ ,\ 3\ ,\ \ol{1}\ ,\ 1_1\ ,\ \ol{4})$};
    
    \draw (0,0) -- (4,0); 

    \draw (0,0) -- (0,1);
    \draw (1,0) -- (1,4); 
    \draw (2,0) -- (2,4);
    \draw (3,0) -- (3,4);
    \draw (4,0) -- (4,4);

    \draw (0,1) -- (1,1);
    \draw (1,4) -- (2,4);
    \draw (2,3) -- (3,3);
    \draw (3,4) -- (4,4);
    
    \node[blk] at (0,1) {       $1_2$};
    \node[blk] at (1,4) {\Large $3$};
    \node[wht] at (2,3) {\Large $\ol{1}$};
    \node[blk] at (3,4) {       $1_1$};
    \node[wht] at (4,0) {\Large $\ol{4}$};
    \node at (0.5,0.5) {+1};
    \node at (1.5,2) {+4};
    \node at (2.5,1.5) {+3};
    \node at (3.5,2) {+4};

    \begin{scope}[shift={(0,-4.5)}]
    \node[blk] (31) at (0,0) {       $1_2$};
    \node[blk] (32) at (1,1) {\Large $3$};
    \node[wht] (33) at (2,2) {\Large $\ol{1}$};
    \node[blk] (34) at (3,3) {       $1_1$};
    \node[wht] (35) at (4,0) {\Large $\ol{4}$};

    \draw[line width=1.4pt, double=white, double distance=1pt, black] (31) -- node[pos=0.5, below=5pt] {1} (35);
    \draw[thick] (32) -- node[pos=0.5, above=5pt] {2} (35);
    \draw[thick] (32) -- node[pos=-0.2, above=10pt] {1} (33);
    \draw[thick] (34) -- node[pos=0.9, above=20pt] {1} (35);
    \end{scope}
    \end{scope}

    \begin{scope}[shift={(0, -14.5)}]
    \node at (1.2 ,6.5) {\Large $\perm_4 = (1_1\ ,\ 1_2\ ,\ 3\ ,\ \ol{1}\ ,\ \ol{4})$};
    
    \draw (0,0) -- (4,0); 

    \draw (0,0) -- (0,1);
    \draw (1,0) -- (1,2); 
    \draw (2,0) -- (2,5);
    \draw (3,0) -- (3,5);
    \draw (4,0) -- (4,4);

    \draw (0,1) -- (1,1);
    \draw (1,2) -- (2,2);
    \draw (2,5) -- (3,5);
    \draw (3,4) -- (4,4);
    
    \node[blk] at (0,1) {       $1_1$};
    \node[blk] at (1,2) {       $1_2$};
    \node[blk] at (2,5) {\Large $3$};
    \node[wht] at (3,4) {\Large $\ol{1}$};
    \node[wht] at (4,0) {\Large $\ol{4}$};
    \node at (0.5,0.5) {+1};
    \node at (1.5,1) {+2};
    \node at (2.5,2.5) {+5};
    \node at (3.5,2) {+4};

    \begin{scope}[shift={(0,-4.5)}]
    \node[blk] (41) at (0,0) {       $1_1$};
    \node[blk] (42) at (1,1) {       $1_2$};
    \node[blk] (43) at (2,2) {\Large $3$};
    \node[wht] (44) at (3,3) {\Large $\ol{1}$};
    \node[wht] (45) at (4,0) {\Large $\ol{4}$};

    \draw[line width=1.4pt, double=white, double distance=1pt, black] (41) -- node[pos=0.5, below=5pt] {1} (45);
    \draw[thick] (42) -- node[pos=0.3, above=3pt] {1} (45);
    \draw[thick] (43) -- node[pos=0.7, above=10pt] {2} (45);
    \draw[thick] (43) -- node[pos=-0.2, above=10pt] {1} (44);
    \end{scope}
    \end{scope}

    \begin{scope}[shift={(12, -14.5)}, ]

    \node[blk] (t1) at (-1.5,-2.6)  {\Large $1_1$};
    \node[blk] (t2) at (-1.5, 2.6)  {\Large $1_2$};
    \node[wht] (t3) at (0,0)        {\Huge  $\ol{4}$};
    \node[blk] (t4) at (3,0)        {\Huge  $3$};
    \node[wht] (t5) at (6,0)        {\Huge  $\ol{1}$};

    \draw[thick] (t1) -- 
    node[mid_auto, swap] {\Large 1} 
    node[color=red, mid_auto] {\Large$e_3$} (t3);
    \draw[thick] (t2) -- 
    node[mid_auto, swap] {\Large 1} 
    node[color=red, mid_auto] {\Large$e_4$} (t3);
    \draw[thick] (t3) -- 
    node[pos=0.5, below=5pt] {\Large 2} 
    node[color=red, mid_auto] {\Large$e_2$} (t4);
    \draw[thick] (t4) -- 
    node[pos=0.5, below=5pt] {\Large 1} 
    node[color=red, mid_auto] {\Large$e_1$} (t5);
    \end{scope}

\end{tikzpicture}
}
    \caption{The four rooted trees}
    \label{fig:the_four_rooted_trees}
\end{figure}

\newpage
\subsection{Values of Kochetkov's formula}\label{ssec:A2}

We compute the values of Kochetkov's formula for all passports with small total weights $n:= \normbig{\ppf} = 4,5,6,7$. 

Explanations for Table \ref{tab:n7}.  
To save space, we shall denote a full passport by its canonically corresponding trivial passport. 
The rows and columns represent the weights at black and white vertices, respectively. They are sorted in decreasing lexicographic order. 
For example, when $n=4$, the entry ``$2$'' in row 3, column 2 means $\Tree(\ppf)=2$ for $\ppf=(2_1\ 2_2\ \ol{3}_1\ \ol{1}_1)$. 

When exchanging the color of all vertices, the formula returns the same value. So each table is actually symmetric. Only the lower triangle parts are exhibited.

\begin{table}[htbp]
\footnotesize  
\rowcolors{1}{gray!20}{gray!20}
\setlength{\tabcolsep}{2.5pt}
\raggedright

\hspace*{1cm}%
\begin{minipage}[t]{\dimexpr\textwidth-1cm\relax}

\begin{minipage}[b]{0.35\textwidth}
    \raggedright
    \begin{tabular}{c|*{5}{c}}
    \hline
    $n=4$ & $4$ & $3\,1$ & $2^2$ & $2\,1^2$ & $1^4$ \\
    \hline
    $4$         & 1 & \multicolumn{4}{c}{} \\
    $3\,1$      & 1 & 1 & \multicolumn{3}{c}{} \\
    $2^2$       & 1 & 2 & 0 & \multicolumn{2}{c}{} \\
    $2\,1^2$    & 2 & 2 & 2 & 0 & \multicolumn{1}{c}{} \\
    $1^4$       & 6 & 0 & 0 & 0 & 0 \\
    \hline
    \end{tabular}
\end{minipage}
\hfill
\begin{minipage}[b]{0.6\textwidth}
    \raggedright
    \begin{tabular}{c|*{7}{c}}
    \hline
    $n=5$ & $5$ & $4\,1$ & $3\,2$ & $3\,1^2$ & $2^2\,1$ & $2\,1^3$ & $1^5$ \\
    \hline
    $5$         & 1 & \multicolumn{6}{c}{} \\

    $4\,1$      & 1 & 1 & \multicolumn{5}{c}{} \\

    $3\,2$      & 1 & 2 & 1 & \multicolumn{4}{c}{} \\

    $3\,1^2$    & 2 & 2 & 4 & 4 & \multicolumn{3}{c}{} \\

    $2^2\,1$    & 2 & 4 & 2 & 4 & 4 & \multicolumn{2}{c}{} \\

    $2\,1^3$    & 6 & 6 & 6 & 0 & 0 & 0 & \multicolumn{1}{c}{} \\

    $1^5$       & 24 & 0 & 0 & 0 & 0 & 0 & 0 \\
    \hline
    \end{tabular}
\end{minipage}

\vspace{15pt}

\begin{tabular}{c|*{11}{c}}
\hline
$n=6$ & $6$ & $5\,1$ & $4\,2$ & $4\,1^2$ & $3^2$ & $3\,2\,1$ & $3\,1^3$ & $2^3$ & $2^2\,1^2$ & $2\,1^4$ & $1^6$ \\
\hline
$6$         & 1 & \multicolumn{10}{c}{} \\

$5\,1$      & 1 & 1 & \multicolumn{9}{c}{} \\

$4\,2$      & 1 & 2 & 1 & \multicolumn{8}{c}{} \\

$4\,1^2$    & 2 & 2 & 4 & 4 & \multicolumn{7}{c}{} \\

$3^2$       & 1 & 2 & 2 & 6 & 0 & \multicolumn{6}{c}{} \\

$3\,2\,1$   & 2 & 4 & 4 & 8 & 2 & 7 & \multicolumn{5}{c}{} \\

$3\,1^3$    & 6 & 6 & 12 & 12 & 12 & 12 & 0 & \multicolumn{4}{c}{} \\

$2^3$       & 2 & 6 & 0 & 12 & 6 & 6 & 12 & 0 & \multicolumn{3}{c}{} \\

$2^2\,1^2$  & 6 & 12 & 8 & 12 & 8 & 12 & 0 & 12 & 0 & \multicolumn{2}{c}{} \\

$2\,1^4$    & 24 & 24 & 24 & 0 & 24 & 0 & 0 & 0 & 0 & 0 & \multicolumn{1}{c}{} \\

$1^6$       & 120 & 0 & 0 & 0 & 0 & 0 & 0 & 0 & 0 & 0 & 0 \\

\hline
\end{tabular}

\vspace{15pt}

\begin{tabular}{c|*{15}{c}}
\hline
$n=7$ & $7$ & $6\,1$ & $5\,2$ & $5\,1^2$ & $4\,3$ & $4\,2\,1$ & $4\,1^3$ & $3^2\,1$ & $3\,2^2$ & $3\,2\,1^2$ & $3\,1^4$ & $2^3\,1$ & $2^2\,1^3$ & $2\,1^5$ & $1^7$ \\
\hline
$7$         & 1 & \multicolumn{14}{c}{} \\

$6\,1$      & 1 & 1 & \multicolumn{13}{c}{} \\

$5\,2$      & 1 & 2 & 1 & \multicolumn{12}{c}{} \\

$5\,1^2$    & 2 & 2 & 4 & 4 & \multicolumn{11}{c}{} \\

$4\,3$      & 1 & 2 & 2 & 6 & 1 & \multicolumn{10}{c}{} \\

$4\,2\,1$   & 2 & 4 & 4 & 8 & 4 & 11 & \multicolumn{9}{c}{} \\

$4\,1^3$    & 6 & 6 & 12 & 12 & 18 & 24 & 36 & \multicolumn{8}{c}{} \\

$3^2\,1$    & 2 & 4 & 6 & 12 & 2 & 10 & 24 & 4 & \multicolumn{7}{c}{} \\

$3\,2^2$    & 2 & 6 & 2 & 16 & 4 & 8 & 36 & 12 & 4 & \multicolumn{6}{c}{} \\

$3\,2\,1^2$ & 6 & 12 & 14 & 24 & 10 & 24 & 36 & 24 & 24 & 36 & \multicolumn{5}{c}{} \\

$3\,1^4$    & 24 & 24 & 48 & 48 & 48 & 48 & 0 & 48 & 48 & 0 & 0 & \multicolumn{4}{c}{} \\

$2^3\,1$    & 6 & 18 & 6 & 36 & 12 & 24 & 36 & 24 & 12 & 36 & 0 & 36 & \multicolumn{3}{c}{} \\

$2^2\,1^3$  & 24 & 48 & 36 & 48 & 36 & 48 & 0 & 48 & 48 & 0 & 0 & 0 & 0 & \multicolumn{2}{c}{} \\

$2\,1^5$    & 120 & 120 & 120 & 0 & 120 & 0 & 0 & 0 & 0 & 0 & 0 & 0 & 0 & 0 & \multicolumn{1}{c}{} \\

$1^7$       & 720 & 0 & 0 & 0 & 0 & 0 & 0 & 0 & 0 & 0 & 0 & 0 & 0 & 0 & 0 \\
\hline
\end{tabular}

\end{minipage}

\caption{\small Values of Kochetkov's formula for total weight $4,5,6,7$.}
\label{tab:n7}
\end{table}

\newpage\subsection{Notation table} ~ \\

\begin{table}[h]
\makebox[1.06\textwidth][c]{
    \begin{minipage}[t]{0.65\textwidth}
    \begin{tabular}{l|l||}
        & \\ 
        $\RR^*$ & non-zero real numbers \\
        $\Zpos$ & positive integers \\
        \quad $[N]$ & \quad $\{1,2,\dots, N\} \subset \Zpos$ \\
        & \\
        $\pp = (S,\lambda,\wt)$ & passport (Def. \ref{def:passport}) \\ 
        \quad $S$ & \quad index set \\ 
        \quad $\lambda: S \to \Zpos$ & \quad multiplicity function \\
        \quad $\wt : S\to \Rnz$ & \quad weight function \\ 
        \quad $\absbig{\pp}$ & \quad number of vertices \\ 
        \quad $\Vert \pp \Vert$ & \quad total weight \\ 
        \quad $\ppf$ & \quad full $\sim$ \\
        \quad $\pp(S')$ & \quad sub-$\sim$ \\
        \quad $\widetilde{\pp}_F$ & \quad perturbed $\sim$ (Def. \ref{def:perturbed_passport}) \\
        & \\

        $(V, E, \wt_E)$ & WBP-tree (Sec. \ref{ssec:LWBP_tree}) \\ 
        \quad $V = V^+\sqcup V^-$ &\quad black and white vertices \\
        \quad $\wt_E : E \to \Rpos$ &\quad weight on edges \\
        \quad $\wt_V: V \to \Rnz$ &\quad weight on vertices \\
        \quad $\sigma_v:E(v) \to E(v)$ &\quad cyclic action (Def. \ref{def:cyclic_action}) \\
        & \\
        
        $T = (V, E, \wt_E,\Lab)$ & LWBP-tree (Def. \ref{def:labeled_tree}) \\ 
        \quad $\Lab: V \to S$ & \quad labeling of WBP-tree \\
        \quad $\Tree(\pp)$ & \quad trees of given passport \\
        \quad $\Fore(\pp)$ & \quad forests of given passport \\ 
        \quad $\FF M(\ppf)$ & \quad twice marked forests \\ 
        \quad $\TT M(\ppf)$ & \quad twice marked trees \\
        \quad $\TR (\ppf)$ & \quad rooted trees (Def. \ref{def:pos_perm})\\
        & \\ 
    \end{tabular}
    \end{minipage}
    \hfill
    \begin{minipage}[t]{0.65\textwidth}
    \begin{tabular}{l|l}
        $\ptt{p} = \{ \pp_i \}_{i=1}^{n}$ & partition (Def. \ref{def:partition}) \\
        \quad $|\ptt{p}|$ & \quad length of $\sim$ \\
        \quad $\ptt{e}$ & \quad trivial 1-$\sim$ \\
        \quad $\Ptt(\ppf)$ & \quad set of $\sim$ \\
        \quad $\Ptt_n(\ppf)$ & \quad set of $n$-$\sim$ \\ 
        \quad $m(\ppf)$ & \quad maximum $\sim$ length \\ 
        \quad $X(\ptt{p})$ & \quad product in Def. \ref{def:partition} \\ 
        \quad $\ptt{p}^o$ & \quad ordered $\sim$ (Def. \ref{def:ordd_partition}) \\
        \quad $\Ptt_n^o(\ppf)$ & \quad set of ordered $n$-$\sim$ \\ 
        & \\

        $\PP(\ppf)$ & permutations (Sec. \ref{ssec:perm_to_tree}) \\
        \quad $\perm = (s_i)_{i=1}^N$ &\quad a permutation \\
        \quad $H_i$ &\quad cumulative sum (Constr. \ref{cnstr:combing_step1}) \\
        \quad $\PP_+(\ppf)$ &\quad positive $\sim$ (Def. \ref{def:pos_perm}) \\
        \quad $\PT(\ppf)$ &\quad tree $\sim$ (Def. \ref{def:PT}) \\
        \quad $\PT_+(\ppf)$ &\quad positive tree $\sim$ (Def. \ref{def:pos_perm}) \\
        \quad $\PP_{\geq0}(\ppf)$ &\quad non-negative $\sim$ (Lem. \ref{lem:Pos_refined_passport_and_formula}) \\
        & \\

        $\comb$ & combing map (Def. \ref{def:comb}) \\ 
        \quad $Rec_i$ & \quad vertical rectangle \\ 
        \quad $e$ &\quad horizontal rectangle \\ 
        \quad $R(\perm)$ & \quad region of $\perm$ \\ 
        \quad $G(\perm)$ & \quad forest from $\perm$ \\ 
        & \\
        
        $\fold$ & folding map (Constr. \ref{cnstr:fold2}) \\
        \quad $\mathbf{e}$ & \quad horizontal rectangle \\ 
        \quad $\widehat{R}(T;a,b)$& \quad region from tree \\ 
        \quad $\perm(T;a,b)$ & \quad image of $\fold$ \\ 
    \end{tabular}
    \end{minipage}
}
\label{tab:notation}
\end{table}

\end{appendix}

\end{document}